\documentclass[a4paper]{amsart}

\usepackage{texdraw}
\usepackage{graphicx}
\usepackage{amsmath}
\usepackage{amssymb}
\usepackage{subfigure}
\usepackage{url}
\usepackage{multirow}

\input xy
\xyoption{all}
\xyoption{arc}

\makeindex

\renewcommand{\mod}{\operatorname{mod}\nolimits}

\newcommand{\rad}{\operatorname{rad}\nolimits}

\newcommand{\id}{\operatorname{id}\nolimits}
\newcommand{\Mod}{\operatorname{Mod}\nolimits}

\newcommand{\rep}{\operatorname{rep}\nolimits}
\newcommand{\Rep}{\operatorname{Rep}\nolimits}

\newcommand{\op}{\operatorname{op}\nolimits}

\newcommand{\Hom}{\operatorname{Hom}\nolimits}

\renewcommand{\dim}{\operatorname{dim}\nolimits}

\newcommand{\rl}{\operatorname{rl}\nolimits}
\newcommand{\Char}{\operatorname{char}\nolimits}

\newcommand{\exend}{\begingroup \renewcommand{\qed}{\hfill \ensuremath{\triangle}} \qed \endgroup}

\newtheorem{theorem}{Theorem}[section]

\newtheorem{corollary}[theorem]{Corollary}

\newtheorem{lemma}[theorem]{Lemma}

\newtheorem{proposition}[theorem]{Proposition}

\theoremstyle{definition}
\newtheorem{example}[theorem]{Example}

\begin{document}


\title[Structure Theorems for Basic Algebras]{Structure Theorems for Basic Algebras}

\author{Carl Fredrik Berg}
\address{Carl Fredrik Berg \\ Institutt for matematiske fag, NTNU\\ 7491 Trondheim\\ Norway}
\curraddr{StatoilHydro R\&D Centre \\ Arkitekt Ebbells veg 10, Rotvoll\\ 7053 
Trondheim \\ Norway}
\email{carlpaatur@hotmail.com}


\subjclass[2000]{}



\begin{abstract}
A basic finite dimensional algebra over an algebraically closed field $k$ is isomorphic to a quotient of a tensor algebra by an admissible ideal. The category of left modules over the algebra is isomorphic to the category of representations of a finite quiver with relations. In this article we will remove the assumption that $k$ is algebraically closed to look at both perfect and non-perfect fields. We will introduce the notion of species with relations to describe the category of left modules over such algebras. If the field is not perfect, then the algebra is isomorphic to a quotient of a tensor algebra by an ideal that is no longer admissible in general. This gives hereditary algebras isomorphic to a quotient of a tensor algebra by a non-zero ideal. We will show that these non-zero ideals correspond to cyclic subgraphs of the graph associated to the species of the algebra. This will lead to the ideal being zero in the case when the underlying graph of the algebra is a tree.
\end{abstract}

\maketitle

\tableofcontents

It is well known that a basic finite dimensional algebra $\Lambda$
over an algebraically closed field $k$ is isomorphic to a quotient
of a path algebra $k \Gamma$ of a finite quiver $\Gamma$.  Moreover
the path algebra $k \Gamma$ is isomorphic to a tensor algebra
\cite[Theorem III.1.9]{ARS}. This was first outlined by Gabriel in
\cite{Ga2}, and he gave a concise proof in \cite[Section 4.3]{Ga3}.
From now on we will call this result \emph{Gabriel's structure
theorem for basic finite dimensional algebras over an algebraically
closed field}, or just the \emph{structure theorem}.

In this article we will discuss what happens if the field $k$ is not
algebraically closed. If one tries to follow the proof of Gabriel,
two assumptions on the algebra $\Lambda$ arise; the first is that
$\Lambda$ splits, i.e.\ the natural projection onto the quotient
algebra $\pi \colon \Lambda \to \Lambda / \rad \Lambda$ splits as a
$k$-algebra homomorphism.  Hence there exists an $\epsilon \colon
\Lambda / \rad \Lambda \to \Lambda$ such that $\pi \circ \epsilon
\simeq \id_{\Lambda / \rad \Lambda}$.  Via $\epsilon$ all
$\Lambda$-modules can be viewed as $\Lambda / \rad \Lambda$-modules.
The second assumption is that for any $\epsilon \colon \Lambda /
\rad \Lambda \to \Lambda$ such that $\pi \circ \epsilon \simeq
\id_{\Lambda / \rad \Lambda}$ the short exact sequence
$$0 \to (\rad \Lambda)^2 \to \rad \Lambda \to \rad \Lambda / (\rad
\Lambda)^2 \to 0$$ splits when it via $\epsilon$ is viewed as a
sequence of $\Lambda / \rad \Lambda - \Lambda / \rad
\Lambda$-bimodules. If both these assumptions are fulfilled, we
get a generalization of the structure theorem; $\Lambda$
is isomorphic to a quotient of a tensor algebra, and this tensor
algebra is constructed from a species associated to the algebra
$\Lambda$ \cite[Proposition 4.1.10]{Ben}.

The main topic of this article is investigating what happens if we
remove the second assumption above. We will show that $\Lambda$ is
still a quotient of a tensor algebra associated to a species,
however not the same species as was used before: To get a
morphism from a tensor algebra onto the algebra $\Lambda$ we have to
take the tensor algebra over a larger bimodule than the one used when both assumptions were fulfilled. Therefore the kernel
of this morphism is no longer an admissible ideal in the tensor
algebra, which gives some interesting observations in the case
$\Lambda$ is hereditary, e.g.\ the hereditary algebra $\Lambda$ need
no longer be a tensor algebra. Examples of such algebras are already
known, and we will use an example from \cite{DR3} to highlight this
property.

\vspace{0.5cm}

In the first section we will introduce notions used throughout this
article. Readers experienced with finite dimensional algebras will
likely be familiar with all the notions introduced.

The second section introduces species with relations.  Since the
existing literature does not treat this concept in detail, we will
give a fairly thorough discussion of it here.

In the third section we will give the structure theorem for finite
dimensional basic split algebras for which the sequence
$$0 \to (\rad \Lambda)^2 \to \rad \Lambda \to \rad \Lambda / (\rad
\Lambda)^2 \to 0$$ splits, using species with relations.  Most
results in this section are similar to well known results, but the
usage of species with relations is however not common.  We will also
show that finite dimensional basic algebras over perfect fields
satisfy the assumptions above.

The fourth section gives a structure theorem for finite
dimensional basic split algebras.  This structure theorem is a
generalization of Gabriel's structure theorem, however it is not a
generalization of the structure theorem given in section three.

In the last section we will describe hereditary basic finite dimensional split algebras. In contrast to the case for algebras over algebraically closed fields, the species of these hereditary algebras might have non-zero relations corresponding to subquivers for which the underlying graph contains cycles.


\section{Preliminaries}

This first section will be used to introduce notions we will need in
the rest of this article.

Throughout this section we will assume that $\Lambda$ is an
indecomposable finite dimensional algebra over a field $k$. Since
$\Lambda$ then is artinian, we know $\Lambda / \mathfrak{r}$ is
semisimple \cite[Theorem 4.14]{La1}, where $\mathfrak{r} = \rad
\Lambda$ is the Jacobson radical of the algebra $\Lambda$. Since the
algebra is finite dimensional, the radical $\mathfrak{r}$ is
nilpotent\index{nilpotent}, i.e.\ $\mathfrak{r}^n = (0)$ for a large
enough $n \in \mathbb{N}$.

When we view $\Lambda$ as a left module over itself, it can be
written as a direct sum of indecomposable projective left
$\Lambda$-modules ${_\Lambda \Lambda} = \oplus_{i \in I} P_i$. When
the $P_i$ are pairwise non-isomorphic projective $\Lambda$-modules
we say that $\Lambda$ is \emph{basic}\index{basic algebra}. A finite
dimensional algebra $\Lambda$ is always Morita equivalent to a basic
finite dimensional algebra. Thus, if we are interested in the module
category of an algebra, we can always reduce the question to a basic
algebra. If we assume that $\Lambda$ is basic, then $\Lambda /
\mathfrak{r} \simeq \oplus_{i \in I} P_i / \rad P_i \simeq \oplus_{i
\in I} D_i$, where $D_i$ are division rings \cite[Theorem
3.5.4]{DK}. Since $\Lambda$ was assumed to be a $k$-algebra, $k$
will act centrally on the division rings $D_i$, i.e.\ for all
$\lambda \in k$ and all $\lambda' \in D_i$ we have that
$\lambda\lambda' = \lambda'\lambda$. The direct sum
$\oplus_{i \in I} D_i$ contains $k$ as a subfield.

We say that a finite dimensional $k$-algebra $\Lambda$ is
\emph{elementary}\index{elementary algebra} if $\Lambda /
\mathfrak{r} \simeq \oplus_{i=1}^n k$, i.e.\ isomorphic as a
$k$-algebra to a finite direct sum of copies of $k$. If $k$ is
algebraically closed\index{algebraically closed}, i.e.\ $k$ has no
proper algebraic extension, then the only finite dimensional
division algebra over $k$ is $k$ itself, so basic implies elementary
when $k$ is algebraically closed and $\Lambda$ is finite
dimensional.

Let $R$ be a ring. An element $e \in R$ is called an
\emph{idempotent}\index{idempotent!idempotent element} if $e^2 = e$.
We call two idempotents $e$ and $f$
\emph{orthogonal}\index{idempotent!orthogonal idempotents} if $ef =
0 = fe$, and we call an idempotent $e$
\emph{primitive}\index{idempotent!idempotent primitive} if $e \not=
f + g$ where $f$ and $g$ are nonzero orthogonal idempotents. A set
of pairwise orthogonal primitive idempotents $\{ e_1, e_2, \dots e_n
\}$ in a ring $R$ will be called
\emph{complete}\index{idempotent!complete set of idempotents} if
$e_1 + e_2 + \cdots + e_n = 1_R$, where $1_R$ is the multiplicative
identity of $R$. Let $\Lambda$ be a finite dimensional algebra and
let $1_\Lambda$ be the multiplicative identity in $\Lambda$. Since
${_\Lambda \Lambda} = \oplus_{i \in I} P_i$, we will have $1_\Lambda
= \Sigma_{i \in I} e_i$, where $e_i \in P_i$.  It is easy to see
that the elements $\{ e_i \}_{i \in I}$ are pairwise orthogonal
idempotents, and it can be shown that they are primitive
\cite[Corollary 7.4]{AF}. Hence we have a complete set of
pairwise orthogonal idempotents $\{ e_i \}_{i \in I}$ in $\Lambda$,
and we have ${_\Lambda \Lambda} = \oplus_{i \in I} \Lambda e_i$
where $\Lambda e_i \simeq P_i$ \cite[Corollary 7.3]{AF}.

The quotient $\mathfrak{r} / \mathfrak{r}^2$ has a natural $\Lambda
/ \mathfrak{r}$-bimodule structure by letting $(\lambda +
\mathfrak{r}) (r + \mathfrak{r}^2) (\lambda' + \mathfrak{r}) =
\lambda r \lambda' + \mathfrak{r}^2$, where $\lambda, \lambda' \in
\Lambda$ and $r \in \mathfrak{r}$.  Obviously $\mathfrak{r} /
\mathfrak{r}^2 = (\Lambda / \mathfrak{r})( \mathfrak{r} /
\mathfrak{r}^2)( \Lambda / \mathfrak{r})$. By using the
decomposition $\Lambda / \mathfrak{r} \simeq \oplus_{i \in I} D_i$,
we get a decomposition $\mathfrak{r} / \mathfrak{r}^2 \simeq
\oplus_{i,j \in I} D_j \mathfrak{r} / \mathfrak{r}^2 D_i =
\oplus_{i,j \in I} ({_j M_i})$.\sloppy

Let $V$ be a $\Sigma-\Sigma$-bimodule, where $\Sigma$ is a ring. We
write $V^{(n)}$ for the $n$-fold tensor product $V \otimes_\Sigma V
\otimes_\Sigma \dots \otimes_\Sigma V$, and we let $V^{(0)} =
\Sigma$. The \emph{tensor ring}\index{tensor ring} of $\Sigma$ and
$V$ is defined as the graded ring $T(\Sigma, V) = V^{(0)} \oplus
V^{(1)} \oplus V^{(2)} \oplus \cdots$, where multiplication $V^{(n)}
\times V^{(m)} \to V^{(n+m)}$ is given using the tensor product over
$\Sigma$: For $\Sigma_{i \in I} a_{i,1} \otimes \cdots \otimes
a_{i,n} \in V^{(n)}$ and $\Sigma_{j \in J} b_{j,1} \otimes \cdots
\otimes b_{j,m} \in V^{(m)}$ we let $(\Sigma_{i \in I} a_{i,1}
\otimes \cdots \otimes a_{i,n}) (\Sigma_{j \in J} b_{j,1} \otimes
\cdots \otimes b_{j,m}) = \Sigma_{i \in I, j \in J} a_{i,1} \otimes
\cdots \otimes a_{i,n} \otimes b_{j,1} \otimes \cdots \otimes
b_{j,m}$.  If $\Sigma$ is a $k$-algebra and $k$ acts centrally on
$V$, then $k$ acts centrally on all $V^{(n)}$, and we can view
$T(\Sigma, V)$ as a $k$-algebra by letting $l(\Sigma_{i \in I}
a_{i,1} \otimes \cdots \otimes a_{i,n}) = \Sigma_{i \in I} l a_{i,1}
\otimes \cdots \otimes a_{i,n}$ for $l \in k$ and $\Sigma_{i \in I}
a_{i,1} \otimes \cdots \otimes a_{i,n} \in V^{(n)}$, where the
multiplication $l a_{i,1}$ is taken using the $k$-algebra structure
of $V$.  When we view the tensor ring $T(\Sigma, V)$ as a
$k$-algebra, we call it the \emph{tensor algebra}\index{tensor
algebra} of $\Sigma$ and $V$.

A morphism between two $k$-algebras is called a \emph{$k$-algebra
homomorphism}\index{algebra homomorphism}\index{$k$-algebra
homomorphism} if it is a ring homomorphism when the algebras are
viewed as rings, and at the same time a $k$-homomorphism when the
algebras are viewed as $k$-modules.

We say that a $k$-algebra $\Lambda$ \emph{splits} or that $\Lambda$
is a \emph{split algebra}\index{split algebra} if the natural
projection $\pi \colon \Lambda \to \Lambda / \mathfrak{r}$ splits in
the sense that there exists a $k$-algebra homomorphism $\epsilon
\colon \Lambda / \mathfrak{r} \to \Lambda$ such that $\pi \epsilon =
\id_{\Lambda / \mathfrak{r}}$. Observe that $\epsilon$ is not
unique. By the Wedderburn-Malcev theorem $\Lambda$ splits when $\sup
\{ n \mid H^n_R(\Lambda, M) \not= (0)$ for some $\Lambda$-bimodule
$M \} \leq 1$, where $H^n_R(\Lambda, M)$ is the $n$'th Hochschild
cohomology module of $\Lambda$ with coefficients in $M$ \cite[p.\
209]{Pi}.  This happens in particular when $k$ is a perfect field,
as we will see in Proposition \ref{prop:perfect_gives_p1p2}.

We end this section with an outline of the proof of Gabriel's
structure theorem, which says that an elementary (or equivalently
basic) finite dimensional algebra $\Lambda$ over an algebraically
closed field $k$ is isomorphic to a quotient of the path algebra $k
\Gamma$ of a finite quiver $\Gamma$ \cite[Theorem III.1.9]{ARS}.
Take the tensor algebra $T$ of the $k$-algebra $\Lambda /
\mathfrak{r}$ and the $\Lambda / \mathfrak{r} - \Lambda /
\mathfrak{r}$-bimodule $\mathfrak{r} / \mathfrak{r}^2$.  This is a
$k$-algebra, and there exists a $k$-algebra epimorphism $\tilde{f}
\colon T \to \Lambda$. The tensor algebra $T$ is isomorphic to the
path algebra $k\Gamma$ of a finite quiver $\Gamma$, so we get a
$k$-algebra epimorphism $k \Gamma \simeq T \to \Lambda$, which shows
that $\Lambda$ is isomorphic to a quotient of a path algebra.

\section{Species}

In this section we
want to introduce the notion of \emph{species with relations}.
Similar ideas have been used before (e.g.\ in \cite{As} under the name
\emph{bounden species}), but then only as an ideal in the algebra
corresponding to the species. In contrast, we want to introduce
relations for a species in a similar fashion as was done for
relations in the path algebra. Throughout this article we will work
with left modules, there are dual definitions and proofs in the
right module case.

A \emph{species}\index{species} (also
known as a \emph{modulated quiver}) $\mathcal{S} = (D_i, {_j
M_i})_{i,j \in I}$ is a set of division rings $D_i$ and $D_j -
D_i$-bimodules $_j M_i$ such that $\Hom_{D_i}({_j M_i}, D_i) \simeq
\Hom_{D_j}({_j M_i}, D_j)$ as $D_i-D_j$-bimodules. We say a species
$(D_i, {_j M_i})_{i,j \in I}$ is a
\emph{$k$-species}\index{$k$-species} if all the division rings
$D_i$ are finite dimensional over a common central subfield $k$, all
the bimodules $_j M_i$ are finite dimensional over $k$, and $\lambda
m = m \lambda$ for all $\lambda \in k$ and $m \in {_jM_i}$. So for a
$k$-species we have $\Hom_{D_i}({_j M_i}, D_i)
\simeq \Hom_{D_j}({_j M_i}, D_j)$ as $D_i-D_j$-bimodules, hence our definition of $k$-species is similar
to the definition of a $k$-species given in \cite{Ga2} and
\cite{Ri5}. All species we will work with in connection with finite
dimensional $k$-algebras are $k$-species.

To visualize a species $\mathcal{S} = (D_i, {_jM_i})_{i,j \in I}$ we
draw a quiver where we use the division rings $D_i$ as vertices, and
for each non-zero bimodule $_j M_i$ we draw an arrow starting in
$D_i$ and ending in $D_j$ and index the arrow using the bimodule $_j
M_i$. For example the species given by the division rings $D_1 =
\mathbb{R}$ and $D_2 = \mathbb{C}$, and the bimodules ${_2 M_1} =
\mathbb{C}$ and ${_1 M_2} = (0)$, will be drawn as
$$\xymatrix{ \mathbb{R} \ar[r]^{\mathbb{C}} & \mathbb{C}}$$ We will call the division rings $D_i$ the vertices of
$\mathcal{S}$, and when we view $D_i$ as a vertex we will sometimes
just call it $i$.

The \emph{underlying quiver}\index{species!underlying quiver}
$Q_\mathcal{S}$ of $\mathcal{S}$ is the quiver with vertices $i \in
I$ and arrows $i \to j$ for all ${_j M_i} \not= (0)$.  We say that a
species $\mathcal{S}$ is finite if the underlying quiver
$Q_\mathcal{S}$ if finite, and we say that $\mathcal{S}$ is without
oriented cycles if there are no oriented cycles $i \to i_1 \to
\cdots i_{n-1} \to i$ in $Q_\mathcal{S}$.

A \emph{representation}\index{species!representation of a species}
$V = (V_i, {_j\phi_i})$ over a species $\mathcal{S} = (D_i,
{_jM_i})_{i,j \in I}$ is a set of left $D_i$-modules $V_i$ together
with morphisms
$${_j\phi_i} \colon {_j M_i} \otimes_{D_i} V_i  \to V_j$$ where ${_j
M_i} \otimes_{D_i} V_i$ is viewed as a left $D_j$-module.
Composition of morphisms ${_k\phi_j} \circ {_j\phi_i} = {_k\phi'_i}
\colon ({_k M_j} \otimes_{D_j} {_j M_i}) \otimes_{D_i} V_i  \to V_k$
is given by ${_k\phi'_i} (({_k m_j} \otimes_{D_j} {_j m_i})
)\otimes_{D_i} v_i) = {_k\phi_j} ({_k m_j} \otimes_{D_j} ({_j\phi_i}
( {_j m_i} \otimes_{D_i} v_i)))$, where ${_k m_j} \in {_k M_j}, {_j
m_i} \in {_j M_i}$ and $v_i \in V_i$.

Let ${_j \mathcal{P}_i}$ be the set of all paths $p$ in
$Q_\mathcal{S}$ which start in the vertex $i$ and end in $j$, and
let $n_p$ be the length of the path $p$.  The vertices in the path
$p$ will be denoted $p(l)$ for $0 \leq l \leq n_p$ in such a way
that $p$ is the path $i = p(0) \to p(1) \to \cdots \to p(n_p-1) \to
p(n_p) = j$. We then have a $D_j - D_i$ bimodule
$${_j \mathcal{M}_i} = \bigoplus_{p \in {_j \mathcal{P}_i}} {_j
M_{p(n_p-1)}} \otimes_{D_{p(n_p-1)}} \cdots \otimes_{D_{p(1)}}
{_{p(1)} M_i}$$ Using the composing of morphisms described above,
from the morphisms $\{ _j \phi_i \}_{i,j \in I}$ we induce a unique
morphism ${_j f_i} \colon {_j \mathcal{M}_i} \otimes_{D_i} V_i \to
V_j$ for each pair $i, j \in I$.

The set of representations $V = (V_i, {_j\phi_i})$ over a species
$\mathcal{S} = (D_i, {_jM_i})_{i,j \in I}$ gives rise to an abelian
category $\Rep \mathcal{S}$\index{$\Rep$} in which a morphism
$\alpha \colon V = (V_i, {_j\phi_i}) \to (V'_i, {_j\phi'_i}) = V'$
is a set of $D_i$-linear maps $\alpha_i \colon V_i \to V'_i$ such
that the diagram
$$\xymatrix{ {_j M_i} \otimes_{D_i} V_i \ar[d]^{_j\phi_i}
\ar[r]^{1 \otimes \alpha_i} & {_j M_i} \otimes_{D_i} V'_i
\ar[d]^{_j\phi'_i}
\\ V_j \ar[r]^{\alpha_j} & V'_j}$$ commutes for all $i,j \in I$.
The full abelian subcategory of $\Rep \mathcal{S}$ consisting of all
representations $V$ for which all $V_i$ are finite dimensional as
vector spaces over $k$ will be denoted $\rep
\mathcal{S}$\index{$\rep$}.

For a species $\mathcal{S} = (D_i, {_jM_i})_{i,j \in I}$ we let the
\emph{tensor algebra} $T(\mathcal{S})$ of $\mathcal{S}$ be the
tensor ring $T(D, M) = T( \oplus_{i \in I} D_i, \oplus_{i,j \in I}
{_j M_i})$. Here we view $M = \oplus_{i,j \in I} {_j M_i}$ as a $D =
\oplus_{i \in I} D_i$ bimodule the natural way.  Since $D$ is a
$k$-algebra and $k$ acts centrally on $M$, we know that
$T(\mathcal{S})$ is a $k$-algebra. Let $J$ denote the ideal
$\oplus_{i \geq 1} M^{(i)}$ in $T(\mathcal{S})$.  Then
$T(\mathcal{S}) / J \simeq D = \oplus_{i \in I} D_i$ is semisimple.

The relation between representations over a species and the modules
over the corresponding tensor algebra is similar to the
correspondence between quiver representations and modules over the
corresponding tensor algebra, as shown in the following proposition.
Although this next proposition is known, we still include the
proof here since ideas from it will be used repeatedly in the rest
of this article.

\begin{proposition}\cite[Proposition 10.1]{DR1}\label{prop:species_Rep_is_tensor_Mod}
Let $\mathcal{S} = (D_i, {_jM_i})_{i,j \in I}$ be a finite
$k$-species. Then the category $\Rep \mathcal{S}$ and the category
$\Mod T(\mathcal{S})$ of left $T(\mathcal{S})$-modules are
equivalent.
\end{proposition}
\begin{proof}
We want to define two functors $$F \colon \Rep \mathcal{S}
\rightleftarrows \Mod T(\mathcal{S}) \colon G$$ such that $G \circ F
\simeq \id_{\Rep \mathcal{S}}$ and $F \circ G \simeq \id_{\Mod
T(\mathcal{S})}$.

We start with $G$.  Let $V \in \Mod T(\mathcal{S})$.  Since $D =
\oplus_{i \in I} D_i$ is a subring of $T(\mathcal{S})$, we can view
$V$ as a left $D$-module.  Since $1_{T(\mathcal{S})} \in D$ we have
$D V = V$, therefore $V = (\oplus_{i \in I} D_i) V = \oplus_{i \in
I} (D_i V) = \oplus_{i \in I} V_i$ where $V_i = D_i V$, hence the
central idempotents in $D$ decompose $V$ as a left $D$-module.

View $M = \oplus_{i,j \in I} ({_j M_i})$ as a $D - D$-bimodule, then
${_j M_i} D_l = (0)$ for $l \not= i$. Since $V$ is a
$T(\mathcal{S})$-module and $M$ is a subset of $T(\mathcal{S})$, we
get a morphism $\phi' \colon M \times V \to V$ where $\phi'(m,v) =
mv$ by using the $T(\mathcal{S})$-module structure on $V$. This
morphism is $D$-biadditive, so it gives rise to an additive morphism
$\phi \colon M \otimes_D V \to V$ where $\phi(m \otimes v) = mv$. We
view $\phi$ as a left $D$-module morphism using that $M$ is a left
$D$-module. Since ${_j M_k} \otimes_{D} V_l = {_j M_k} \otimes_{D}
(D_l V) = (0)$ for all $k \not= l$, we see that $M \otimes_D V
\simeq \oplus_{i,j \in I} ({_j M_i} \otimes_{D_i} V_i)$. Observe
that $\phi ({_j M_i} \otimes_{D_i} V_i) = \phi (D({_j M_i}
\otimes_{D_i} V_i)) = \phi (D_j ({_j M_i} \otimes_{D_i} V_i)) = D_j
\phi ({_j M_i} \otimes_{D_i} V_i) \subseteq D_j V = V_j$.  Let ${_j
\phi_i} = \phi \mid_{{_j M_i} \otimes_{D_i} V_i} \colon {_j M_i}
\otimes_{D_i} V_i \to V_j$.  Define $G$ on objects by letting $G(V)
= (V_i, {_j \phi_i})_{i,j \in I}$.

Let $\alpha \colon V \to V' \in \Mod T(\mathcal{S})$. Since $\alpha$
is a morphism of left $T(\mathcal{S})$-modules, it is also a
morphism of left $D = \oplus_{i \in I} D_i$-modules, and then in
particular a left $D_i$-module morphism for every $i \in I$, hence
$\alpha(V_i) \subseteq V'_i$.  Let $\alpha_i = \alpha \mid_{V_i}
\colon V_i \to V'_i$, and let $G(\alpha) = \{ \alpha_i \}_{i \in
I}$. To see that $\{ \alpha_i \}_{i \in I}$ is a map of
$\mathcal{S}$ representations, we need to check that ${_j \phi'_i}
\circ (1 \otimes \alpha_i) = \alpha_j \circ {_j \phi_i}$ for all
$i,j \in I$. Using that $\alpha$ is a morphism of
$T(\mathcal{S})$-modules, we have \begin{align*}{_j \phi'_i} \circ
(1 \otimes \alpha_i)(m \otimes v) &= {_j \phi'_i} (m \otimes
\alpha_i(v)) = m \alpha_i(v) = m \alpha(v) \\ & = \alpha(mv)  =
\alpha_j(mv) = \alpha_j \circ {_j \phi_i} (m \otimes v) \end{align*}
for $m \in {_j M_i}$ and $v \in V_i$. This shows that $G \colon \Mod
T(\mathcal{S}) \to \Rep \mathcal{S}$ is a functor.

We then have to construct a functor $F \colon \Rep \mathcal{S} \to
\Mod T(\mathcal{S})$. For an object $(V_i, {_j \phi_i})_{i,j \in I}$
in $\Rep \mathcal{S}$, let $V = \oplus_{i \in I} V_i$. Let $D$
operate on $V$ the obvious way, namely using the left $D_i$
structure on $V_i$, and letting $D_j V_i = (0)$ for all $j \not= i$.
View ${_j \phi_i}$ as a morphism $M \otimes_D V \to V$ by letting
${_j \phi_i} \mid_{{_l M_k} \otimes_{D_k} V_k} = 0$ for $k \not= i$
and $l\not=j$ where $k, i, l, j \in I$. Then we can define $\phi =
\Sigma_{i,j \in I} {_j \phi_i} \colon M \otimes_D V \to V$. Let $M$
operate on $V$ using $\phi$, hence for $m \in M$ and $v \in V$, let
$m v = \phi (m \otimes v)$. By induction we define $M^{(n)} \times V
= M \otimes_D \cdots \otimes_D M \times V \to V$ using the morphism
$\phi^{(n)} \colon M^{(n)} \otimes_D V \to V$ where $\phi^{(n)} =
\phi(1 \otimes_D \phi^{(n-1)})$ and $\phi^{(1)} = \phi$. Let $F
((V_i, {_j \phi_i})_{i,j \in I}) = V$ where $V$ has this
$T(\mathcal{S})$-module structure.

Let $\{ \alpha_i \}_{i \in I} \colon (V_i, {_j \phi_i})_{i,j \in I}
\to (V'_i, {_j \phi'_i})_{i,j \in I}$ be a morphism in $\Rep
\mathcal{S}$. View $\alpha_i$ as a morphism on $V$ by letting
$\alpha_i \mid_{V_j} = 0$ for $j \not=i \in I$, and let $\alpha =
\Sigma_{i \in I} \alpha_i \colon V = \oplus_{i \in I} V_i \to
\oplus_{i \in I} V'_i = V'$. Since all $\alpha_i$ are left
$D_i$-linear, we only need to show that $\alpha(mv) = m\alpha(v)$
for $m \in \oplus_{i \geq 1} M^{(i)}$. Since multiplication by an
element in $\oplus_{i \geq 1} M^{(i)}$ is induced by the
multiplication of elements in $M$, this is true if $\alpha(mv) =
m\alpha(v)$ for $m \in M$.  Invoking that $\alpha$ is $D$-linear and
$M = \oplus_{i,j \in I} {_j M_i}$ as a $D-D$-bimodule, what one
needs to show is that $\alpha(mv) = m\alpha(v)$ for $m \in {_j
M_i}$.  Using that $\{ \alpha_i \}_{i \in I}$ is a morphism in $\Rep
\mathcal{S}$ we see that $\alpha(mv) = \alpha_j ({_j \phi_i}(m
\otimes v)) = {_j \phi'_i} \circ (1 \otimes \alpha_i)(m \otimes v) =
{_j \phi'_i} (m \otimes \alpha_i(v)) = m\alpha(v)$. This shows that
$F \colon \Rep \mathcal{S} \to \Mod T(\mathcal{S})$ is a functor.

Observing that $G \circ F \simeq \id_{\Rep \mathcal{S}}$ and $F
\circ G \simeq \id_{\Mod T(\mathcal{S})}$, we have proven the
proposition.
\end{proof}

\begin{corollary}\label{cor:species_rep_is_tensor_mod}
Let $\mathcal{S} = (D_i, {_jM_i})_{i,j \in I}$ be a finite
$k$-species. Then the category $\rep \mathcal{S}$ and the category
$\mod T(\mathcal{S})$ of finite dimensional left
$T(\mathcal{S})$-modules are equivalent.
\end{corollary}
\begin{proof}
From Proposition \ref{prop:species_Rep_is_tensor_Mod} we have two functors $$F \colon
\Rep \mathcal{S} \rightleftarrows \Mod T(\mathcal{S}) \colon G$$
such that $G \circ F \simeq \id_{\Rep \mathcal{S}}$ and $F \circ G
\simeq \id_{\Mod T(\mathcal{S})}$.  We want to show that $F
\mid_{\rep \mathcal{S}} \subset \mod T(\mathcal{S})$ and $G
\mid_{\mod T(\mathcal{S})} \subset \rep \mathcal{S}$.

Therefore let $(V_i, {_j \phi_i}) \in \rep \mathcal{S}$.  Then $\dim_k V_i < \infty$, and since
$\mathcal{S}$ is finite we get that $V= \oplus_{i \in I} V_i$ is
finite dimensional over $k$ too.  Hence $F((V_i, {_j \phi_i})) = V \in \mod T(\mathcal{S})$.

Now let $V \in \mod T(\mathcal{S})$.  Then $\dim_k V < \infty$, therefore $V_i = D_iV$ is finite dimensional over $k$ too,
so $G(V) = (V_i, {_j \phi_i}) \in \rep \mathcal{S}$.
\end{proof}

A \emph{relation}\index{species!relation}\index{relation} $\sigma$
of a species $\mathcal{S} = (D_i, {_j M_i})_{i,j \in I}$ is a sum
$\sigma = g_1 + \cdots + g_n$ of elements $g_l = g_{{l, n_l}}
\otimes \cdots \otimes g_{l,1} \in {_{i_{(n_l,l)}} M_{i_{(n_l
-1,l)}}} \otimes_{D_{i_{(n_l-1,l)}}} \cdots \otimes_{D_{i_{(l,1)}}}
{_{i_{(l,1)}} M_{i_{(l,0)}}}$ where $i_{(l,n_l)}=b$ and
$i_{(l,0)}=a$ for all $1 \leq l \leq n$.  We will write the relation
$\sigma$ as $_b \sigma_a$ when we want to emphasize that it starts
in $a$ and ends in $b$. Let $\rho = \{ \sigma_t \}_{t \in T}$ be a
set of relations, where the different elements $\sigma_t$ possibly
start and end in different vertices. We call the pair $(\mathcal{S},
\rho)$ a \emph{species with relations}\index{relation!species}.
Define $T(\mathcal{S}, \rho) = T(\mathcal{S}) / \langle \rho
\rangle$ where $\langle \rho \rangle$ is the ideal in
$T(\mathcal{S})$ generated by the elements $\{ \sigma_t \}_{t \in
T}$.  Also, define $\Rep (\mathcal{S}, \rho)$ as the category of
representations $V \in \Rep \mathcal{S}$ for which ${_j f_i}
\mid_{\langle {_j \sigma_i} \rangle \otimes_{D_i} V_i} = 0$ whenever
there is an element ${_j \sigma_i} \in \rho$, where $\langle {_j
\sigma_i} \rangle$ is the subspace of $_j \mathcal{M}_i$ generated
by ${_j \sigma_i}$ as a $D_j - D_i$-bimodule. Let $\rep
(\mathcal{S}, \rho) = \Rep (\mathcal{S}, \rho) \cap \rep
\mathcal{S}$.

The next proposition is a generalization of \cite[Proposition
1.7]{ARS}.

\begin{proposition}\label{prop:species_with_rel_is_tensor_with_ideal}
Let $\mathcal{S}$ be a finite $k$-species, and $\rho$ a set of
relations. Then the category $\Rep (\mathcal{S}, \rho)$ and the
category $\Mod (T(\mathcal{S})/\langle \rho \rangle)$ of left
$T(\mathcal{S})/\langle \rho \rangle$-modules are equivalent.
\end{proposition}
\begin{proof}
Recall from Proposition \ref{prop:species_Rep_is_tensor_Mod} that we
have mutually inverse equivalences
$$F \colon \Rep \mathcal{S} \rightleftarrows \Mod T(\mathcal{S})
\colon G$$ We want to show that these functors induce an equivalence
between $\Rep (\mathcal{S}, \rho)$ and $\Mod (T(\mathcal{S})/\langle
\rho \rangle)$.

\sloppy Let $(V_i, {_j \phi_i})_{i,j \in I} \in \Rep (\mathcal{S},
\rho)$. Since ${_j f_i} \mid_{\langle {_j \sigma_i} \rangle
\otimes_{D_i} V_i} = 0$ for every element ${_j \sigma_i} \in \rho$,
we have $\langle {_j \sigma_i} \rangle F((V_i, {_j \phi_i})_{i,j \in
I}) = (0)$, where in the last equation $\langle {_j \sigma_i}
\rangle$ is the ideal in $T(\mathcal{S})$ generated by ${_j
\sigma_i}$. Therefore $\langle \rho \rangle F((V_i, {_j
\phi_i})_{i,j \in I}) = (0)$, so $F((V_i, {_j \phi_i})_{i,j \in I})
\in \Mod (T(\mathcal{S})/\langle \rho \rangle)$.

On the other hand, let $V \in \Mod (T(\mathcal{S})/\langle \rho
\rangle)$.  For ${_j \sigma_i} \in \rho$, we have that $(0) =
\langle {_j \sigma_i} \rangle V = \langle {_j \sigma_i} \rangle
V_i$, which implies that ${_j f_i} \mid_{\langle {_j \sigma_i}
\rangle \otimes_{D_i} V_i} = 0$ for the morphism ${_j f_i} \colon
{_j \mathcal{M}_i} \otimes V_i \to V_j$ in $G(V)$, where $\langle
{_j \sigma_i} \rangle$ in the last equation is the subset of ${_j
\mathcal{M}_i}$ generated by ${_j \sigma_i}$ as a
$D_j-D_i$-bimodule. Hence $G(V) \in \Rep (\mathcal{S}, \rho)$.
\end{proof}

\begin{corollary}\label{cor:species_with_rel_is_tensor_with_ideal}
Let $\mathcal{S}$ be a finite $k$-species, and $\rho$ a set of
relations. Then the category $\rep (\mathcal{S}, \rho)$ and the
category $\mod (T(\mathcal{S})/\langle \rho \rangle)$ of finite
dimensional left $T(\mathcal{S})/\langle \rho \rangle$-modules are
equivalent.
\end{corollary}
\begin{proof}
This follows from Corollary \ref{cor:species_rep_is_tensor_mod} and
Proposition \ref{prop:species_with_rel_is_tensor_with_ideal}.
\end{proof}

Let $\Lambda = (T(\mathcal{S})/\langle \rho \rangle)$ where
$\mathcal{S} = (D_i, {_j M_i})_{i,j \in I}$.  In the category of
left $\Lambda$-modules, the projective modules are $P_i = \Lambda
D_i$ for $i \in I$, where $\Lambda / \mathfrak{r} \simeq \oplus_{i
\in I} D_i$.  Observe that there is a one-to-one correspondence
between the vertices $I$ and the indecomposable projective
representations of $(\mathcal{S}, \rho)$.

\section{Basic Algebras over Perfect Fields}\label{section:perfect_field}

In this section, let $\Lambda$ be a finite dimensional basic algebra
over a field $k$ (not necessarily algebraically closed).  We want to
investigate algebras $\Lambda$ that are split, i.e.\ the natural
projection onto the quotient algebra $\pi \colon \Lambda \to \Lambda
/ \mathfrak{r}$ splits in the sense that there exists a $k$-algebra
homomorphism $\epsilon \colon \Lambda / \mathfrak{r} \to \Lambda$
such that $\pi \epsilon = \id_{\Lambda / \mathfrak{r}}$.

Assume an algebra $\Lambda$ is split.  Recall that the $k$-algebra
homomorphism $\epsilon \colon \Lambda / \mathfrak{r} \to \Lambda$
such that $\pi \epsilon = \id_{\Lambda / \mathfrak{r}}$ is not
unique. Using $\epsilon \colon \Lambda/\mathfrak{r} \to \Lambda$ we
can view $\Lambda / \mathfrak{r} \simeq \oplus_{i \in I} D_i$ as a
subalgebra of $\Lambda$, and we can identify the division rings
$D_i$ with their image in $\Lambda$ under $\epsilon$. Observe that
the subalgebra structure of $\Lambda / \mathfrak{r}$ and the
identification of $D_i$ with a subset of $\Lambda$ is dependent on
the choice of $\epsilon$.

We say that a split algebra $\Lambda$ is \emph{$\mathfrak{r}$-split}
if for any $k$-algebra homomorphism $\epsilon \colon \Lambda /
\mathfrak{r} \to \Lambda$ such that $\pi \epsilon = \id_{\Lambda /
\mathfrak{r}}$, the short exact sequence
$$0 \to \mathfrak{r}^2 \to \mathfrak{r} \to \mathfrak{r} /
\mathfrak{r}^2 \to 0$$ splits when we via $\epsilon$ view the
sequence as a sequence of $\Lambda / \mathfrak{r} - \Lambda /
\mathfrak{r}$-bimodules.

In this section we will give a structure theorem for finite
dimensional $\mathfrak{r}$-split basic algebras.  At the end of this
section we will show that all finite dimensional basic $k$-algebras
over perfect fields $k$ are $\mathfrak{r}$-split, hence they give
rise to a large class of examples of $\mathfrak{r}$-split algebras.

To reach this goal we need a slight reformulation of \cite[Lemma
III.1.2]{ARS}:

\begin{lemma}\label{lemma:old_tensor_map}
Let $\Sigma$ be a $k$-algebra and $V$ a $\Sigma - \Sigma$-bimodule.
Let $\Lambda$ be a $k$-algebra and $f \colon \Sigma \oplus V \to
\Lambda$ a morphism such that $f \mid_{\Sigma} \colon \Sigma \to
\Lambda$ is a $k$-algebra homomorphism and $f \mid_V \colon V \to
\Lambda$ is a $\Sigma - \Sigma$-bimodule morphism when $\Lambda$ is
viewed as a $\Sigma - \Sigma$-bimodule via $f \mid_{\Sigma}$. Then
there exists a unique $k$-algebra homomorphism $\tilde{f} \colon
T(\Sigma, V) \to \Lambda$ such that $\tilde{f} \mid_{\Sigma \oplus
V} = f$.
\end{lemma}
\begin{proof}
Let $\phi \colon V \times V \to \Lambda$ be given by $\phi(v_1, v_2)
= f(v_1)f(v_2)$.  Then for $r \in \Sigma$ we have $\phi(v_1 r, v_2)
= f(v_1 r)f(v_2) = f(v_1) r f(v_2) = f(v_1) f(r v_2) = \phi(v_1, r
v_2)$ since $f \mid_V$ is a $\Sigma - \Sigma$-bimodule morphism.
Also since $f \mid_V$ is a $\Sigma - \Sigma$-bimodule morphism we
get that $\phi(v_1 + v_2, v_3) = \phi(v_1, v_3) + \phi(v_2, v_3)$
and $\phi(v_1, v_2 + v_3) = \phi(v_1, v_2) + \phi(v_1, v_3)$, so
$\phi$ is a $\Sigma$-biadditive morphism. Hence we get an induced
additive morphism $f_2 \colon V \otimes_\Sigma V \to \Lambda$ where
$f_2 (v_1 \otimes v_2) = f(v_1)f(v_2)$.  Using again that $f \mid_V$
is a $\Sigma - \Sigma$-bimodule morphism, we see that for $r \in
\Sigma$ we have $f_2 (r v_1 \otimes v_2) = f(r v_1) f(v_2) = r
f(v_1) f(v_2) = r f_2 (v_1 \otimes v_2)$ and $f_2 (v_1 \otimes v_2
r) = f(v_1) f(v_2 r) = f(v_1) f(v_2) r = f_2 (v_1 \otimes v_2) r$,
hence $f_2$ is a $\Sigma - \Sigma$-bimodule morphism.  By induction
we construct a $\Sigma - \Sigma$-bimodule morphism $f_n \colon
V^{(n)} \to \Lambda$ where $f_n (v_1 \otimes v_2 \otimes \cdots
\otimes v_n) = f(v_1)f(v_2) \cdots f(v_n)$.

Denote an element in $T(\Sigma, V)$ by $\Sigma_{i=0}^\infty v_i$,
where $v_i \in V^{(i)}$ for all $i \in \mathbb{N}$ and $v_i$ is zero
for all but a finite number of $i \in \mathbb{N}$. If we let $f_0 =
f \mid_\Sigma$ and $f_1 = f \mid_V$, we can define $\tilde{f} \colon
T(\Sigma, V) \to \Lambda$ by letting $\tilde{f} (\Sigma_{i=0}^\infty
v_i) = \Sigma_{i=0}^\infty f_i(v_i)$. For two elements $v =
\Sigma_{i=0}^\infty v_i$ and $w = \Sigma_{i=0}^\infty w_i$ in
$T(\Sigma, V)$ we see that
\begin{align*}\tilde{f} (v + w) &= \tilde{f} (\Sigma_{i=0}^\infty
v_i + w_i)\\ &= \Sigma_{i=0}^\infty f_i(v_i+w_i)\\& =
\Sigma_{i=0}^\infty (f_i(v_i)+f_i(w_i))\\& = \tilde{f}
(\Sigma_{i=0}^\infty v_i) + \tilde{f} (\Sigma_{i=0}^\infty w_i) =
\tilde{f} (v) + \tilde{f} (w) \\ \tilde{f} (v w) &= \tilde{f}
(\Sigma_{i=0}^\infty
 \Sigma_{j=0}^i v_j w_{i-j})\\ & = \Sigma_{i=0}^\infty \Sigma_{j=0}^i f_i(v_j w_{i-j}) \\& = \Sigma_{i=0}^\infty \Sigma_{j=0}^i f_j(v_j) f_{i-j}(w_{i-j})\\&
= (\Sigma_{i=0}^\infty f_i(v_i)) (\Sigma_{i=0}^\infty f_i(w_i))
\\& = \tilde{f} (\Sigma_{i=0}^\infty v_i) \tilde{f} (\Sigma_{i=0}^\infty w_i) = \tilde{f} (v)  \tilde{f} (w)
\end{align*} This shows that $\tilde{f}$ is a ring homomorphism.
It is easy to see that $\tilde{f}$ is a $k$-module homomorphism,
hence $\tilde{f}$ is a $k$-algebra homomorphism.  Since $\{ \Sigma,
V \}$ generates $T(\Sigma, V)$, the morphism $\tilde{f}$ unique.
\end{proof}

Since $\Lambda$ is finite dimensional over $k$ we know that
$\mathfrak{r} / \mathfrak{r}^2$ is finitely generated as a $\Lambda
/ \mathfrak{r} - \Lambda / \mathfrak{r}$-bimodule.  We can therefore
find elements $\{ r_1, r_2, \dots, r_m \}$ in $\mathfrak{r}$ such
that their images $\{ \bar{r}_1, \bar{r}_2, \dots, \bar{r}_m \}$ in
$\mathfrak{r} / \mathfrak{r}^2$ generate $\mathfrak{r} /
\mathfrak{r}^2$ as a $\Lambda / \mathfrak{r} - \Lambda /
\mathfrak{r}$-bimodule.  We let $\rl(\Lambda)$ denote the Lowey
length (radical length) of $\Lambda$, i.e.\ the smallest number $n
\in \mathbb{N}$ such that $\mathfrak{r}^n = (0)$.

The following result is a generalization of \cite[Theorem III.1.9
(a)(b)]{ARS}.  The proof of part (a) follows the lines of
\cite[Theorem III.1.9 (a)]{ARS}, and can be found in \cite[Lemma 3.1
(i)]{Li}.  We include the proof here for completeness. Part (b)
could have been proven similarly to the proof of \cite[Theorem
III.1.9 (b)]{ARS}.  We will use another proof since ideas from it
will be used later in Proposition
\ref{prop:tensor_algebra_surjective}.

\begin{proposition}\label{prop:p1p2_tensor_alg_onto}
Let $\Lambda$ be a finite dimensional basic $\mathfrak{r}$-split
$k$-algebra.
\begin{itemize}
\sloppy \item[(a)] Let $\{ r_1, r_2, \dots, r_m \}$ be elements in
$\mathfrak{r}$ such that their images $\{ \bar{r}_1, \bar{r}_2,
\dots, \bar{r}_m \}$ in $\mathfrak{r} / \mathfrak{r}^2$ generate
$\mathfrak{r} / \mathfrak{r}^2$ as a $\Lambda / \mathfrak{r} -
\Lambda / \mathfrak{r}$-bimodule.  Then $\{ D_1, D_2,  \dots, D_n,
r_1, r_2, \dots, r_m \}$ generate $\Lambda$ as a $k$-algebra, where
$\Lambda / \mathfrak{r} \simeq \oplus_{i=1}^n D_i$ is viewed as a
$k$-subalgebra of $\Lambda$ using a $k$-algebra homomorphism
$\epsilon \colon \Lambda / \mathfrak{r} \to \Lambda$ such that $\pi
\epsilon = \id_{\Lambda / \mathfrak{r}}$.
\item[(b)] There is a surjective $k$-algebra homomorphism $\tilde{f} \colon
T (\Lambda / \mathfrak{r}, \mathfrak{r} / \mathfrak{r}^2) \to
\Lambda$ such that $\oplus_{j \geq \rl (\Lambda)} (\mathfrak{r} /
\mathfrak{r}^2)^{(j)} \subset \ker \tilde{f} \subset \oplus_{j \geq
2} (\mathfrak{r} / \mathfrak{r}^2)^{(j)}$.
\end{itemize}
\end{proposition}
\begin{proof}
\begin{itemize}
\item[(a)] We will prove this by induction on the Lowey length of $\Lambda$.
So assume $\rl(\Lambda) = 1$.  Then $\mathfrak{r} = (0)$, so
$\Lambda = \oplus_{i =1}^n D_i$, and $\Lambda$ is obviously
generated as a $k$-algebra by the set $\{ D_1, D_2, \dots, D_n \}$.

\sloppy When $\rl(\Lambda) = 2$, then $\mathfrak{r}^2 = (0)$.  Since
$\Lambda / \mathfrak{r} \simeq \oplus_{i =1}^n D_i$ we have $\Lambda
= \langle D_1, \dots, D_n \rangle + \mathfrak{r}$. Moreover
$\mathfrak{r} \simeq \mathfrak{r} / \mathfrak{r}^2$, so we see that
$\Lambda$ is generated by the set $\{ D_1, \dots, D_n, r_1, \dots,
r_m \}$.

\sloppy Assume (a) is true for algebras with Lowey length $m$, and
assume $\rl(\Lambda) = m +1$. Let $A$ be the $k$-subalgebra of
$\Lambda$ generated by $\{ D_1, D_2,  \dots, D_n, r_1, r_2, \dots,
r_m \}$, and let $x \in \Lambda$. Since $\rl (A / (A \cap
\mathfrak{r}^m) ) = \rl (\Lambda / \mathfrak{r}^m) = m$, we have by
induction that $A / (A \cap \mathfrak{r}^m) \simeq \Lambda /
\mathfrak{r}^m$. Therefore there exists some $y \in A$ such that $x
+ \mathfrak{r}^m = y + (A \cap \mathfrak{r}^m)$. Then $x - y \in
\mathfrak{r}^m$, hence $x-y = \Sigma_{i=1}^s \alpha_i \beta_i$ where
$\alpha_i \in \mathfrak{r}^{m-1}$ and $\beta_i \in \mathfrak{r}$.
Since $\mathfrak{r}^{m-1}  / \mathfrak{r}^m \simeq (A \cap
\mathfrak{r}^{m-1}) / (A \cap \mathfrak{r}^m)$, we get $\alpha_i +
\mathfrak{r}^m = a_i + (A \cap \mathfrak{r}^m)$ where $a_i \in A
\cap \mathfrak{r}^{m-1}$.  This yields $\alpha_i = a_i + a_i'$,
where $a_i \in A \cap \mathfrak{r}^{m-1}$ and $a_i' \in
\mathfrak{r}^m$. Similarly, since $\mathfrak{r}  / \mathfrak{r}^m
\simeq (A \cap \mathfrak{r}) / (A \cap \mathfrak{r}^m)$, we get that
$\beta_i + \mathfrak{r}^m = b_i + (A \cap \mathfrak{r}^m)$ where
$b_i \in A \cap \mathfrak{r}$.  Hence $\beta_i = b_i + b_i'$ where
$b_i \in A \cap \mathfrak{r}$ and $b_i' \in \mathfrak{r}^m$. Then
$\alpha_i \beta_i = (a_i + a_i') (b_i + b_i') = a_i b_i$ since $a_i
b_i' \in \mathfrak{r}^{2m-1} = (0)$, $a_i' b_i \in
\mathfrak{r}^{m+1} = (0)$, and $a_i' b_i' \in \mathfrak{r}^{2m} =
(0)$. This shows that $x-y \in A$, and since $y \in A$ we get that
$x \in A$. Hence $\Lambda = A$, so $\Lambda$ is generated by the set
$\{ D_1, \dots, D_n, r_1, \dots, r_m \}$.

\item[(b)] Let $f\mid_{\Lambda /
\mathfrak{r}} = \epsilon \colon \Lambda / \mathfrak{r} \to \Lambda$
be a lifting of the natural projection $\pi \colon \Lambda \to
\Lambda / \mathfrak{r}$, and view $\Lambda / \mathfrak{r}$ as a
$k$-subalgebra of $\Lambda$ via $\epsilon$. Since $\Lambda$ is
$\mathfrak{r}$-split, the sequence
$$0 \to \mathfrak{r}^2 \to \mathfrak{r} \to \mathfrak{r} /
\mathfrak{r}^2 \to 0$$ splits as a sequence of $\Lambda /
\mathfrak{r} - \Lambda / \mathfrak{r}$-bimodules.  Let $f
\mid_{\mathfrak{r} / \mathfrak{r}^2} \colon \mathfrak{r} /
\mathfrak{r}^2 \to \Lambda$ be given by the splitting map
$\mathfrak{r} / \mathfrak{r}^2 \hookrightarrow \mathfrak{r}$
composed with the inclusion $\mathfrak{r} \hookrightarrow \Lambda$
viewed as a $\Lambda / \mathfrak{r} - \Lambda /
\mathfrak{r}$-bimodule morphism via $\epsilon$.

Using Lemma \ref{lemma:old_tensor_map} we get induced a $k$-algebra
homomorphism $\tilde{f} \colon T (\Lambda / \mathfrak{r},
\mathfrak{r} / \mathfrak{r}^2) \to \Lambda$, and by (a) this
morphism is surjective. Since $\tilde{f} \mid_{\Lambda /
\mathfrak{r} \oplus \mathfrak{r} / \mathfrak{r}^2}$ is a
monomorphism, and the image intersects trivially with
$\mathfrak{r}^2$, we see that $\ker \tilde{f} \subset \oplus_{j \geq
2} (\mathfrak{r} / \mathfrak{r}^2)^{(j)}$. On the other hand,
$\tilde{f} ((\mathfrak{r} / \mathfrak{r}^2)^{(j)}) \subset
\mathfrak{r}^j$, and since $\mathfrak{r}^{\rl(\Lambda)} = (0)$, we
get that $\ker \tilde{f} \supset \oplus_{j \geq \rl (\Lambda)}
(\mathfrak{r} / \mathfrak{r}^2)^{(j)}$.
\end{itemize}
\end{proof}

As mentioned, the result \cite[Theorem III.1.9 (a)(b)]{ARS} is a
special case of Proposition \ref{prop:p1p2_tensor_alg_onto}. In
\cite[Theorem III.1.9 (a)(b)]{ARS} it is assumed that $\Lambda$ is a
finite dimensional basic $k$-algebra where $k$ is an algebraically
closed field, and we will show in Corollary
\ref{cor:closed_is_perfect} that this implies that $\Lambda$ is
$\mathfrak{r}$-split.  The main difference between Proposition
\ref{prop:p1p2_tensor_alg_onto} and \cite[Theorem III.1.9
(a)(b)]{ARS} is that we have replaced the complete set of pairwise
orthogonal idempotents $\{ e_1, e_2, \dots e_n \}$ with the set $\{
D_1, D_2, \dots D_n \}$.  The reason for this change is that in the
algebraically closed case $e_i$ would have generated $D_i$ as a
$k$-algebra, but in our non-algebraically closed case this is no
longer true.

Let $\Lambda$ be a finite dimensional basic $k$-algebra.  Then
$\Lambda / \mathfrak{r} = \oplus_{i \in I} D_i = D$, and
\begin{align*} \mathfrak{r} / \mathfrak{r}^2 & = (\Lambda /
\mathfrak{r})( \mathfrak{r} / \mathfrak{r}^2) (\Lambda /
\mathfrak{r}) = (\oplus_{i \in I} D_i) ( \mathfrak{r} /
\mathfrak{r}^2) (\oplus_{i \in I} D_i)
\\ & = \oplus_{i,j \in I} (D_j ( \mathfrak{r} / \mathfrak{r}^2) D_i)
= \oplus_{i,j \in I} ({_j M_i}) = M \end{align*} Since $k$ sits
inside the center of $\Lambda$, we have $k \subset D$.  Moreover
$\lambda \lambda' = \lambda' \lambda$ and $\lambda m = m \lambda$
for all $\lambda \in k$, $\lambda' \in D_i$ and $m \in {_j M_i}$ for
all $i,j \in I$. Since $\Lambda$ is finite dimensional, we know that
$D_i$ and $_j M_i$ are finite dimensional over $k$ for all $i, j \in
I$. Hence $S_\Lambda = (D_i, {_j M_i})_{i,j \in I}$ is a
$k$-species, and it will be called the \emph{species of
$\Lambda$}\index{species!of $\Lambda$}. Observe that we do not
assume $\Lambda$ is $\mathfrak{r}$-split to define the species of
$\Lambda$.

Remember that we denote the ideal $\oplus_{i \geq 1} M^{(i)}$ in
$T(\mathcal{S})$ by $J$, where $\mathcal{S}$ is the species $(D_i,
{_j M_i})_{i,j \in I}$ and $M = \oplus_{i,j \in I} ({_j M_i})$.

\begin{proposition}\label{prop:p1p2_ring_is_tensor_ring_of_species}
Let $\Lambda$ be a finite dimensional basic $\mathfrak{r}$-split
$k$-algebra.  Then $\Lambda \simeq T(\mathcal{S}_\Lambda)/\langle
\rho \rangle$ with $J^{\rl (\Lambda)} \subset \langle \rho \rangle
\subset J^2$, where $\mathcal{S}_\Lambda$ is the species of
$\Lambda$ and $\rho$ is a set of relations.
\end{proposition}
\begin{proof}
\sloppy Observe that $T(\mathcal{S}_\Lambda) = T(D, M) = T(\Lambda /
\mathfrak{r}, \mathfrak{r} / \mathfrak{r}^2)$.  From Proposition
\ref{prop:p1p2_tensor_alg_onto} we have an epimorphism $\tilde{f}
\colon T(\mathcal{S}_\Lambda) = T(\Lambda / \mathfrak{r},
\mathfrak{r} / \mathfrak{r}^2) \to \Lambda$ with $\oplus_{j \geq \rl
(\Lambda)} (\mathfrak{r} / \mathfrak{r}^2)^{(j)} \subset \ker
\tilde{f} \subset \oplus_{j \geq 2} (\mathfrak{r} /
\mathfrak{r}^2)^{(j)}$.  Let $\rho' = \{ \sigma'_t \}_{t\in T'}$ be
a set of generators for $\ker \tilde{f}$ as an ideal in
$T(\mathcal{S}_\Lambda)$. Since $\mathfrak{r} / \mathfrak{r}^2 = M$
and $\ker \tilde{f} = \langle \rho' \rangle$, we have $J^{\rl
(\Lambda)} = \oplus_{j \geq \rl (\Lambda)} M^{(j)} \subset \langle
\rho' \rangle \subset \oplus_{j \geq 2} M^{(j)} = J^2$.

We want to transfer $\rho'$ into a set of relations.  Let $\{ e_i
\}_{i \in I}$ be a complete set of pairwise orthogonal primitive
idempotents in $\Lambda$.  Let $\sigma'_t \in \rho$, then $\sigma'_t
= 1_\Lambda \sigma'_t 1_\Lambda = (\Sigma_{i \in I} e_i) \sigma'_t
(\Sigma_{i \in I} e_i) = \Sigma_{i,j \in I} e_j \sigma'_t e_i =
\Sigma_{i,j \in I}  ({_j \rho_{t i}})$ where ${_j \rho_{t i}} = e_j
\sigma'_t e_i$. All ${_j \rho_{t i}}$ are sets of relations, so
letting $\rho = \cup_{t \in T'} \cup_{i,j \in I} ({_j \rho_{ti}})$,
we see that $\langle \rho \rangle = \langle \rho' \rangle$ where
$\rho$ is a set of relations.

Using $T(\mathcal{S}_\Lambda)/\langle \rho \rangle \simeq T(\Lambda
/ \mathfrak{r}, \mathfrak{r} / \mathfrak{r}^2) / \ker \tilde{f}
\simeq \Lambda$, we have proven the proposition.
\end{proof}

\sloppy Observe that the set of relations $\rho = \{ \sigma_t
\}_{t\in T}$ in Proposition
\ref{prop:p1p2_ring_is_tensor_ring_of_species} can be chosen to be
finite: The $\rl(\Lambda)$-fold tensor product $M^{(\rl(\Lambda))} =
(\mathfrak{r} /\mathfrak{r}^2)^{(\rl(\Lambda))}$ is finite
dimensional over $k$, so there exists a finite set $\{ \sigma_t
\}_{t\in T'}$ of generators, and this set of generators can be
chosen to consist of relations. Since $(\mathfrak{r}
/\mathfrak{r}^2)^{(\rl(\Lambda))}$ generates $\oplus_{j \geq \rl
(\Lambda)} (\mathfrak{r} / \mathfrak{r}^2)^{(j)}$, the finite set $\{
\sigma_t \}_{t\in T'}$ also generates
$\oplus_{j \geq \rl (\Lambda)} (\mathfrak{r} /
\mathfrak{r}^2)^{(j)}$. Since $\ker \tilde{f} / (\oplus_{j \geq \rl
(\Lambda)} (\mathfrak{r} / \mathfrak{r}^2)^{(j)})$ is finite
dimensional, there exists a finite set of elements $\{ \sigma_t
\}_{t\in T''}$ in $\ker \tilde{f}$ such that the corresponding
elements in $\ker \tilde{f} / (\oplus_{j \geq \rl (\Lambda)}
(\mathfrak{r} / \mathfrak{r}^2)^{(j)})$ is a generating set, and
also $\{ \sigma_t \}_{t\in T''}$ can be chosen to consist of
relations.  Then, letting $T = T' \cup T''$, we know $\{ \sigma_t
\}_{t\in T}$ to be a finite set of relations which generates $\ker
\tilde{f}$.

Using Corollary
\ref{cor:species_with_rel_is_tensor_with_ideal} and Proposition
\ref{prop:p1p2_ring_is_tensor_ring_of_species} we get the following corollary.
\begin{corollary}
Let $\Lambda$ be a finite dimensional basic $\mathfrak{r}$-split
$k$-algebra, and let $\mathcal{S}_\Lambda$ be the species of
$\Lambda$. Then the category $\mod \Lambda$ of finite dimensional
left $\Lambda$-modules is equivalent to $\rep (\mathcal{S}_\Lambda,
\rho)$ where $\rho$ is a set of relations such that $J^{\rl
(\Lambda)} \subset \langle \rho \rangle \subset J^2$.
\end{corollary}

We now want to investigate finite dimensional \emph{hereditary
algebras}\index{hereditary!algebra}, i.e.\ finite dimensional
algebras where all left ideals are projective.  Hereditary algebras
have been studied thoroughly, in particular the rest of the results
in this section are either well known or similar to well known
results. The next two lemmas are restated for completeness.

\begin{lemma}\label{lemma:ideal_in_square_radical}\cite[Lemma III.1.11]{ARS}
If $\Lambda$ is a basic finite dimensional hereditary algebra and
$\mathfrak{a}$ is a non-zero ideal of $\Lambda$ contained in
$\mathfrak{r}^2$, then $\Lambda / \mathfrak{a}$ is not hereditary.
\end{lemma}

\begin{lemma}\label{lemma:map_between_proj_is_mono}\cite[Lemma III.1.12]{ARS}
If $\Lambda$ is a basic finite dimensional hereditary algebra, and
$f \colon P \to Q$ is a non-zero morphism between indecomposable
projective $\Lambda$-modules, then $f$ is a monomorphism.
\end{lemma}

\begin{proposition}
Let $\Lambda$ be a basic finite dimensional hereditary
$\mathfrak{r}$-split $k$-algebra, let $\mathcal{S}_\Lambda$ be the
species of $\Lambda$, and $Q_{\mathcal{S}_\Lambda}$ the underlying
quiver of $\mathcal{S}_\Lambda$.  Then $Q_{\mathcal{S}_\Lambda}$ is
finite and without oriented cycles, and $\Lambda$ is isomorphic to
$T(\mathcal{S}_\Lambda)$.
\end{proposition}
\begin{proof}
We know from Proposition
\ref{prop:p1p2_ring_is_tensor_ring_of_species} that $\Lambda \simeq
T(\mathcal{S}_\Lambda)/\langle \rho \rangle$ with $\langle \rho
\rangle \subset J^2$.  Since $\Lambda$ is hereditary, Lemma
\ref{lemma:ideal_in_square_radical} implies that $\langle \rho
\rangle = (0)$, hence $\Lambda \simeq T(\mathcal{S}_\Lambda)$. Since
$\Lambda$ is finite dimensional, the underlying quiver
$Q_{\mathcal{S}_\Lambda}$ of $\mathcal{S}_\Lambda$ must be finite.

Assume there is an oriented cycle in $Q_{\mathcal{S}_\Lambda}$. Using Lemma \ref{lemma:map_between_proj_is_mono}, this
will give rise to a proper monomorphism from an indecomposable
projective $P$ into itself, which contradicts that the algebra
$\Lambda$ is finite dimensional.
\end{proof}

From the proposition above we see that a basic finite dimensional
hereditary $\mathfrak{r}$-split $k$-algebra is a tensor algebra. Let
us prove the opposite direction.

\begin{lemma}\label{lemma:species_without_relations_hereditary}
Let $\mathcal{S}$ be a $k$-species with underlying quiver
$Q_\mathcal{S}$. If $Q_\mathcal{S}$ is a finite quiver without
oriented cycles, then $T(\mathcal{S})$ is a hereditary finite
dimensional basic $k$-algebra.
\end{lemma}
\begin{proof}
We know that $T(\mathcal{S}) \simeq \oplus_{i \in I} T(\mathcal{S})
D_i \simeq \oplus_{i \in I} P_i$ is a decomposition of
$T(\mathcal{S})$ into a direct sum of indecomposable projectives,
and that $\rad T(\mathcal{S}) = \oplus_{i \in I} (\oplus_{{_j M_i}
\not= 0} (\dim_{D_j} {_j M_i})P_j)$. Hence the radical of
$T(\mathcal{S})$ is projective, which implies that $T(\mathcal{S})$
is hereditary \cite[Theorem 2.35]{La2}.

Since $Q_\mathcal{S}$ is finite and without oriented cycles,
$T(\mathcal{S})$ is finite dimensional.
\end{proof}

We summarize the previous two results in the following theorem.

\begin{theorem}\label{thm:fin_dim_alg_r-split_eq}
Let $\Lambda$ be a basic finite dimensional $\mathfrak{r}$-split
$k$-algebra, then the following are equivalent:
\begin{itemize}
\item[(i)] $\Lambda$ is hereditary
\item[(ii)] $\Lambda$ is isomorphic to a tensor algebra $T(\mathcal{S})$ of a
species $\mathcal{S}$ where the underlying quiver $Q_\mathcal{S}$ is
finite and without oriented cycles
\end{itemize}
\end{theorem}

We now proceed to show there is a large class of finite dimensional
basic $k$-algebras which are $\mathfrak{r}$-split, namely finite
dimensional basic $k$-algebras where $k$ is a perfect field.

\begin{proposition}\label{prop:perfect_gives_p1p2}
If $k$ is a perfect field and $\Lambda$ is a finite dimensional
basic $k$-algebra, then $\Lambda$ is $\mathfrak{r}$-split.
\end{proposition}
\begin{proof}
When $k$ is perfect, $\Lambda$ is split \cite[Corollary 11.6]{Pi}.
We therefore only need to show that the sequence
$$0 \to \mathfrak{r}^2 \to \mathfrak{r} \to \mathfrak{r} /
\mathfrak{r}^2 \to 0$$ splits as a sequence of $\Lambda /
\mathfrak{r} - \Lambda / \mathfrak{r}$-bimodules via any morphism
$\epsilon \colon \Lambda / \mathfrak{r} \to \Lambda$ for which $\pi
\epsilon \simeq \id_{\Lambda / \mathfrak{r}}$ where $\pi \colon
\Lambda \to \Lambda / \mathfrak{r}$ is the natural projection
morphism. Using the decomposition $\Lambda / \mathfrak{r} \simeq
\oplus_{i \in I} D_i$, this is equivalent to showing that the
sequence splits as a sequence of $D_i-D_j$-bimodules for every pair
$i, j \in I$.

A $D_i-D_j$-bimodule $M$ is given by an operation $D_i \times M
\times D_j \to M$.  By duality we get an equivalent operation $\phi
\colon D_i \times D_j^{\op} \times M \to M$.  If $k$ acts centrally
on $M$ and $D_j$, we have $\phi(d_i r, d_j, m) = (d_i r) m d_j = d_i
m (r d_j) = d_i m (d_j r) = \phi (d_i, r d_j, m)$.  Hence the
operation $\phi$ is $k$-biadditive on $D_i \times D_j^{\op}$, so it
gives rise to an additive morphism $D_i \otimes_k D_j^{\op} \times M
\to M$. This way we can view $M$ as a $D_i \otimes_k
D_j^{\op}$-module. Since $k$ acts centrally on all the objects in
the short exact sequence and also on $D_i$ and $D_j$, the sequence
of $D_i-D_j$-bimodules can be viewed as a sequence of $D_i \otimes_k
D_j^{\op}$-modules.

Since $k$ is perfect and $D_j^{\op}$ is a finite extension of $k$,
we know that $D_i \otimes_k D_j^{\op}$ is a semisimple $k$-algebra
\cite[Theorem 5.3.6]{DK}. Therefore $D_i \otimes_k
D_j^{\op}$-modules are projective \cite[Theorem 4.13]{Ro}, so the
sequence splits as a sequence of $D_i \otimes_k D_j^{\op}$-modules,
hence it splits as a sequence of $D_i - D_j$-modules.
\end{proof}

\begin{corollary}\label{cor:closed_is_perfect}
If $k$ is an algebraically closed field, and $\Lambda$ is a finite
dimensional basic $k$-algebra, then $\Lambda$ is
$\mathfrak{r}$-split.
\end{corollary}
\begin{proof}
A field that is algebraically closed is in particular perfect, so
this is a direct consequence of Lemma \ref{prop:perfect_gives_p1p2}.
\end{proof}

We will now summarize what we know about basic finite dimensional
algebras over perfect fields.

\begin{theorem}\label{theorem:all_for_P1_P2}
Let $\Lambda$ be a finite dimensional basic $k$-algebra where $k$ is
a perfect field, let $\mathcal{S}_\Lambda$ be the species of
$\Lambda$, and let $Q_{\mathcal{S}_\Lambda}$ be the underlying
quiver of $\mathcal{S}_\Lambda$.  Then the following holds:
\begin{itemize}
\sloppy \item[(a)] Let $\{ r_1, r_2, \dots, r_m \}$ be elements in
$\mathfrak{r}$ such that their images $\{ \bar{r}_1, \bar{r}_2,
\dots, \bar{r}_m \}$ in $\mathfrak{r} / \mathfrak{r}^2$ generate
$\mathfrak{r} / \mathfrak{r}^2$ as a $\Lambda / \mathfrak{r} -
\Lambda / \mathfrak{r}$-bimodule.  Then $\{ D_1, D_2,  \dots, D_n,
r_1, r_2, \dots, r_m \}$ generate $\Lambda$ as a $k$-algebra, where
$\Lambda / \mathfrak{r} \simeq \oplus_{i=1}^n D_i$.
\item[(b)] There is a surjective $k$-algebra homomorphism $\tilde{f} \colon
T (\Lambda / \mathfrak{r}, \mathfrak{r} / \mathfrak{r}^2) \to
\Lambda$ such that $\oplus_{j \geq \rl (\Lambda)} (\mathfrak{r} /
\mathfrak{r}^2)^{(j)} \subset \ker \tilde{f} \subset \oplus_{j \geq
2} (\mathfrak{r} / \mathfrak{r}^2)^{(j)}$.
\item[(c)] $\Lambda \simeq T(\mathcal{S}_\Lambda)/\langle
\rho \rangle$ with $J^{\rl (\Lambda)} \subset \langle \rho \rangle
\subset J^2$, where $\rho$ is a set of relations on
$\mathcal{S}_\Lambda$.
\item[(d)] The category $\mod \Lambda$ is equivalent to
a category $\rep (\mathcal{S}_\Lambda, \rho)$, where $\rho$ is a set
of relations on $\mathcal{S}_\Lambda$ such that $J^{\rl (\Lambda)}
\subset \langle \rho \rangle \subset J^2$.
\item[(e)] $\Lambda$ is hereditary if and only if $Q_{\mathcal{S}_\Lambda}$ is finite without oriented
cycles and $\Lambda$ is isomorphic to $T(\mathcal{S}_\Lambda)$.
\end{itemize}
\end{theorem}

\begin{example}
As we know, not all finite dimensional algebras are split, and we
will now give an example of such an algebra built on an example in
\cite[p.\ 99]{Ben}. Let $\mathbb{F}_2$ be the Galois field
consisting of two elements, and let $k = \mathbb{F}_2 (x) = \{
\frac{f}{g} \mid f, g \in \mathbb{F}_2[x], g \not= 0 \}$ be the
field of rational functions over $\mathbb{F}_2$ with indeterminate
$x$. Observe that $k$ is not a perfect field.  We want to
investigate the ring $\Lambda = k[y,z] /(z^2, y^2 - x- z)$. Note
that this is a finite dimensional algebra over $k$, actually $\dim_k
\Lambda = 4$, where $\{1, y, z, yz \}$ is a basis for $\Lambda$ as a
$k$-algebra. The Jacobson radical $\mathfrak{r}$ of $\Lambda$ is the
ideal $\langle z \rangle$. This shows $\Lambda / \mathfrak{r} \simeq
k[y] / (y^2 - x) \simeq \mathbb{F}_2 (t)$ where in the last ring we
have $t^2 = x$. Now $\mathbb{F}_2 (t)$ is a finite field extension
of $k$, and $\dim_k \mathbb{F}_2 (t) = 2$ when we view $\mathbb{F}_2
(t)$ as a $k$-algebra.

We want to check if $\Lambda$ splits, hence we want to try to
construct a $k$-algebra homomorphism $\epsilon \colon \Lambda /
\mathfrak{r} \to \Lambda$. Since $\Lambda / \mathfrak{r}$ as a
$k$-algebra is generated by $\{1, y\}$, we only need to define
$\epsilon$ on the element $y$.  To get a morphism we need $x =
\epsilon(x) = \epsilon(y^2) = \epsilon(y) \epsilon(y)$. Since there
are no solutions to the equation $u^2 - x$ for an indeterminate $u$
in the ring $\Lambda$, this is impossible. Hence there are no
$k$-algebra homomorphisms $\epsilon \colon \Lambda / \mathfrak{r}
\to \Lambda$, so $\Lambda$ does not split. \exend
\end{example}

\section{Basic Finite Dimensional Split Algebras}

In this section we want to give a structure theorem for finite
dimensional basic split $k$-algebras.  The proof of this result will
have many similarities with the proof for the structure theorem for
finite dimensional basic $\mathfrak{r}$-split $k$-algebras
(Proposition \ref{prop:p1p2_ring_is_tensor_ring_of_species}), but
there are at the same time important differences. Even though
$\mathfrak{r}$-split algebras in particular are split algebras, the
structure theorem we are about to give for split algebras is not a
generalization of the structure theorem for $\mathfrak{r}$-split
algebras.  On the other hand, it is a generalization of Gabriel's
structure theorem for finite dimensional basic $k$-algebras where
$k$ is algebraically closed.

In contrast to the case when $\Lambda$ was $\mathfrak{r}$-split, it
is in general no longer true that the sequence
$$0 \to \mathfrak{r}^2 \to \mathfrak{r} \to \mathfrak{r} /
\mathfrak{r}^2 \to 0$$ splits when viewed as a sequence of $\Lambda
/ \mathfrak{r} - \Lambda / \mathfrak{r}$-bimodules via any
$k$-algebra homomorphism $\epsilon \colon \Lambda / \mathfrak{r} \to
\Lambda$ such that $\pi \epsilon \simeq \id_{\Lambda /
\mathfrak{r}}$ for the natural projection $\pi \colon \Lambda \to
\Lambda / \mathfrak{r}$ (we know that such a $k$-algebra
homomorphism $\epsilon$ exists since $\Lambda$ is assumed to be
split). The splitting of the above sequence is used in the
construction of a tensor algebra mapping onto $\Lambda$. We will use
a similar construction in this section, but we will only need a free
$\Lambda / \mathfrak{r} - \Lambda / \mathfrak{r}$-bimodule $F$ such
that $F$ maps onto $\mathfrak{r} / \mathfrak{r}^2$.  Using that $F$
is free we get a lifting as shown in the following commutative
diagram.
$$\xymatrix{ &&& F \ar@{->>}[d] \ar@{.>}[dl] \\ 0 \ar[r] & \mathfrak{r}^2 \ar[r] & \mathfrak{r} \ar[r] &
\mathfrak{r} / \mathfrak{r}^2 \ar[r] & 0}$$ We will replace the
splitting morphism $\mathfrak{r} / \mathfrak{r}^2 \to \mathfrak{r}$
with the morphism $F \to \mathfrak{r}$ given by this lifting, we can
use an argument similar to the one in the $\mathfrak{r}$-split case
to construct a tensor algebra mapping onto $\Lambda$.

One natural choice for a free $\Lambda / \mathfrak{r} - \Lambda /
\mathfrak{r}$-bimodule mapping onto $\mathfrak{r} / \mathfrak{r}^2$
could have been $\Lambda / \mathfrak{r} \otimes_k \mathfrak{r} /
\mathfrak{r}^2 \otimes_k \Lambda / \mathfrak{r}$, but this module
will turn out to be unnecessarily large. Since we have $\Lambda /
\mathfrak{r} \simeq \oplus D_i$ we only need to find a free
$D_j-D_i$-module ${_j F_i}$ mapping onto $D_j (\mathfrak{r} /
\mathfrak{r}^2) D_i = {_j M_i}$ for all $i, j \in I$.  Using that
${_j F_i}$ is free, we will get a lifting resulting in the follwing
commutative diagram of $D_j-D_i$-modules.
$$\xymatrix{ &&& {_j F_i} \ar@{->>}[d] \ar@{.>}[dl]
\\ 0 \ar[r] & D_j (\mathfrak{r}^2) D_i \ar[r] & D_j (\mathfrak{r}) D_i \ar[r] &
D_j (\mathfrak{r} / \mathfrak{r}^2) D_i \ar[r] & 0}$$ A natural
choice for such a module is ${_j \tilde{M}_i} = D_j \otimes_k D_j
(\mathfrak{r} / \mathfrak{r}^2) D_i \otimes_k D_i = D_j \otimes_k
({_j M_i}) \otimes_k D_i$, where $\mathfrak{r} / \mathfrak{r}^2
\simeq \oplus_{i,j \in I} ({_j M_i})$.  An onto $D_j-D_i$-module
morphism ${_j g_i} \colon {_j \tilde{M}_i} \to {_j M_i}$ is given by
${_j g_i}(d_j \otimes {_j m_i} \otimes d_i) = d_j ({_j m_i}) d_i$.
Taking the direct sum we get the $\Lambda / \mathfrak{r} - \Lambda /
\mathfrak{r}$-bimodule $\tilde{M} = \oplus_{i,j \in I} ({_j
\tilde{M}_i})$.  Let $g \colon \tilde{M} \to \mathfrak{r} /
\mathfrak{r}^2$ be given by $g = \Sigma_{i,j \in I} ({_j g_i})$
where we extend the functions ${_j g_i}$ to $\tilde{M}$ by letting
${_j g_i} \mid_{{_l \tilde{M}_k}} = 0$ for $k \not= i$ and $l \not=
j$ in $I$. Then $g$ maps onto $\mathfrak{r} / \mathfrak{r}^2$ and
factors through $\mathfrak{r}$, as shown in the diagram below.
$$\xymatrix{ &&& \tilde{M} \ar@{->>}[d] \ar@{.>}[dl] \\ 0 \ar[r] & \mathfrak{r}^2 \ar[r] & \mathfrak{r} \ar[r] &
\mathfrak{r} / \mathfrak{r}^2 \ar[r] & 0}$$

We define the \emph{enlarged species of
$\Lambda$}\index{species!enlarged species of
$\Lambda$}\index{enlarged!species of $\Lambda$} to be the species
$\tilde{\mathcal{S}}_\Lambda = (D_i, {_j \tilde{M}_i})_{i,j \in I}$.
Since $\tilde{M}$ is a $\Lambda / \mathfrak{r}$-bimodule we are able
to define the \emph{enlarged tensor algebra of
$\Lambda$}\index{enlarged!tensor algebra of $\Lambda$}\index{tensor
algebra!enlarged tensor algebra of $\Lambda$} as $\tilde{T} (\Lambda
/ \mathfrak{r}, \mathfrak{r} / \mathfrak{r}^2) = T(\Lambda /
\mathfrak{r}, \tilde{M})$.  The tensor algebra of the species
$\tilde{\mathcal{S}}_\Lambda$ will then be the same as the enlarged
tensor algebra of $\Lambda$ since $T(\tilde{\mathcal{S}}_\Lambda) =
T(\oplus_{i\in I} D_i, \oplus_{i,j \in I} ({_j \tilde{M}_i})) =
T(\Lambda / \mathfrak{r}, \tilde{M}) = \tilde{T} (\Lambda /
\mathfrak{r}, \mathfrak{r} / \mathfrak{r}^2)$. The $n$-fold tensor
product $\tilde{M}^{(n)}$ can be simplified the following way:
\begin{align*} & \tilde{M}^{(n)} = \tilde{M}
\otimes_{\Lambda/\mathfrak{r}} \cdots \otimes_{\Lambda/\mathfrak{r}}
\tilde{M} \\ & = (\oplus_{i,j \in I} D_i \otimes_k {_i M_j}
\otimes_k D_j) \otimes_{\Lambda/\mathfrak{r}} \cdots
\otimes_{\Lambda/\mathfrak{r}} (\oplus_{i,j \in I} D_i \otimes_k {_i
M_j} \otimes_k D_j) \\ &\simeq \oplus_{i_0, i_1, \dots, i_n \in I}
(D_{i_0} \otimes_k {_{i_0} M_{i_1}} \otimes_k D_{i_1} \otimes_k
{_{i_1} M_{i_2}} \otimes_k \cdots \otimes_k D_{i_{n-1}} \otimes_k
{_{i_{n-1}} M_{i_n}} \otimes_k D_{i_n})
\end{align*}

If $k$ is an algebraically closed field then $D_i \simeq k$ for all
$i \in I$, hence $$\tilde{M} = \oplus_{i,j \in I} D_j \otimes_k {_j
M_i} \otimes_k D_i \simeq \oplus_{i,j \in I} k \otimes_k {_j M_i}
\otimes_k k \simeq \oplus_{i,j \in I} ({_j M_i}) = M =
\mathfrak{r}/\mathfrak{r}^2$$ Hence in the algebraically closed case
$\tilde{\mathcal{S}}_\Lambda = \mathcal{S}_\Lambda$ and $\tilde{T}
(\Lambda / \mathfrak{r}, \mathfrak{r} / \mathfrak{r}^2) = T(\Lambda
/ \mathfrak{r}, \mathfrak{r} / \mathfrak{r}^2)$.

The next proposition is similar to Proposition
\ref{prop:p1p2_tensor_alg_onto}.

\begin{proposition}\label{prop:tensor_algebra_surjective}
Assume that $\Lambda$ is a finite dimensional basic split
$k$-algebra.
\begin{itemize}
\sloppy \item[(a)] Let $\{ r_1, r_2, \dots, r_m \}$ be elements in
$\mathfrak{r}$ such that their images $\{ \bar{r}_1, \bar{r}_2,
\dots, \bar{r}_m \}$ in $\mathfrak{r} / \mathfrak{r}^2$ generate
$\mathfrak{r} / \mathfrak{r}^2$ as a $\Lambda / \mathfrak{r} -
\Lambda / \mathfrak{r}$-bimodule.  Then $\{ D_1, D_2,  \dots, D_n,
r_1, r_2, \dots, r_m \}$ generate $\Lambda$ as a $k$-algebra, where
$\Lambda / \mathfrak{r} \simeq \oplus_{i=1}^n D_i$.
\item[(b)] There is a surjective $k$-algebra homomorphism $\tilde{f} \colon
\tilde{T} (\Lambda / \mathfrak{r}, \mathfrak{r} / \mathfrak{r}^2)
\to \Lambda$ such that $\oplus_{j \geq \rl (\Lambda)}
\tilde{M}^{(j)} \subset \ker \tilde{f} \subset \oplus_{j \geq 1}
\tilde{M}^{(j)}$.
\end{itemize}
\end{proposition}
\begin{proof}
\begin{itemize}
\item[(a)] Similar to Proposition \ref{prop:p1p2_tensor_alg_onto}.

\item[(b)]  Fix a map $\epsilon \colon \Lambda/ \mathfrak{r} \to \Lambda$ such that $\pi \epsilon \simeq \id_D$ for the
natural projection $\pi \colon \Lambda \to \Lambda /\mathfrak{r}$, and let $g$ be as described above this proposition. From the
earlier discussion $g$ lifts to a composition $ph$ where $h \colon
\tilde{M} \to \mathfrak{r}$ is a $\Lambda/ \mathfrak{r}-\Lambda/ \mathfrak{r}$-bimodule morphism and $p$ is
the natural projection $\mathfrak{r} \twoheadrightarrow M$ viewed as
a $\Lambda/ \mathfrak{r}-\Lambda/ \mathfrak{r}$-bimodule morphism via the fixed lifting $\epsilon$.

Construct the morphism $f \colon \Lambda/ \mathfrak{r} \oplus \tilde{M} \to \Lambda$ by
letting $f \mid_{\Lambda/ \mathfrak{r}} = \epsilon$ and $f \mid_{\tilde{M}} = ih$ where
$i$ is the inclusion $\mathfrak{r} \hookrightarrow \Lambda$ viewed
as a $\Lambda/ \mathfrak{r}-\Lambda/ \mathfrak{r}$-bimodule morphism via the lifting $\epsilon$.  Then,
since $\epsilon$ is a $k$-algebra homomorphism and $ih$ is a
$\Lambda/ \mathfrak{r}-\Lambda/ \mathfrak{r}$-bimodule morphism via $\epsilon$, we can use Lemma
\ref{lemma:old_tensor_map} to find a $k$-algebra homomorphism
$\tilde{f} \colon \tilde{T}(\Lambda / \mathfrak{r}, \mathfrak{r} /
\mathfrak{r}^2) \simeq T(\Lambda/ \mathfrak{r}, \tilde{M}) \to \Lambda$.  By part (a)
this morphism is surjective.

Since $\tilde{f} \mid_{\Lambda / \mathfrak{r}} = f\mid_{\Lambda /
\mathfrak{r}} = \epsilon$ is a monomorphism the image intersects
trivially with $\mathfrak{r}$. Using $\tilde{f} (\tilde{M}) \subset
\mathfrak{r}$, we then get $\ker \tilde{f} \subset \oplus_{j \geq 1}
\tilde{M}^{(j)}$. On the other hand, since $\tilde{f}
(\tilde{M}^{(j)}) \subset \mathfrak{r}^j$ and
$\mathfrak{r}^{\rl(\Lambda)} = (0)$, we get that $\ker \tilde{f}
\supset \oplus_{j \geq \rl (\Lambda)} \tilde{M}^{(j)}$.
\end{itemize}
\end{proof}

Even though this result is similar to Proposition
\ref{prop:p1p2_tensor_alg_onto}, and then also to \cite[Theorem
III.1.9]{ARS}, there is an important difference.  In \cite[Theorem
III.1.9 (b)]{ARS} one assumes $\Lambda$ to be a finite dimensional
basic $k$-algebra where $k$ is algebraically closed, and shows there
is a surjective $k$-algebra homomorphism $\tilde{f}' \colon T
(\Lambda / \mathfrak{r}, \mathfrak{r} / \mathfrak{r}^2) \to \Lambda$
with $\oplus_{j \geq \rl (\Lambda)} (\mathfrak{r} /
\mathfrak{r}^2)^j \subset \ker \tilde{f}' \subset \oplus_{j \geq 2}
(\mathfrak{r} / \mathfrak{r}^2)^j$.  This is a special case of
Proposition \ref{prop:p1p2_tensor_alg_onto}(b), where one assumes
that $\Lambda$ is a finite dimensional basic $\mathfrak{r}$-split
$k$-algebra, and shows there is a surjective $k$-algebra
homomorphism $\tilde{f}'' \colon T (\Lambda / \mathfrak{r},
\mathfrak{r} / \mathfrak{r}^2) \to \Lambda$ with $\oplus_{j \geq \rl
(\Lambda)} (\mathfrak{r} / \mathfrak{r}^2)^j \subset \ker
\tilde{f}'' \subset \oplus_{j \geq 2} (\mathfrak{r} /
\mathfrak{r}^2)^j$. Hence Proposition
\ref{prop:tensor_algebra_surjective} differs from these two results
since the kernel of $\tilde{f}'$ and $\tilde{f}''$ sits inside
$\oplus_{j \geq 2} \mathfrak{r} / \mathfrak{r}^2$, while
$\ker(\tilde{f}) \subset \oplus_{j \geq 1} \tilde{M}^j$ (note which
sets the direct sum is taken over). This difference is not
surprising, since $\tilde{M}$ usually is a larger module than
$\mathfrak{r} / \mathfrak{r}^2$. If $k$ is an algebraically closed
field, then $\tilde{M} \simeq \mathfrak{r} / \mathfrak{r}^2$, so
$\tilde{f} \mid_{\tilde{M}} = \tilde{f}' \mid_{\mathfrak{r} /
\mathfrak{r}^2} \colon \mathfrak{r} / \mathfrak{r}^2 \to
f(\mathfrak{r} / \mathfrak{r}^2) = f(\tilde{M})$ is a $\Lambda /
\mathfrak{r} - \Lambda / \mathfrak{r}$-bimodule monomorphism
intersecting trivially both with the image of $\tilde{f} \mid_{D}$
and $\mathfrak{r}^2$, so we get $\ker(\tilde{f}) \subset \oplus_{j
\geq 2} \tilde{M}^j$. This shows that \cite[Theorem III.1.9
(b)]{ARS} is a special case of Proposition
\ref{prop:tensor_algebra_surjective}(b). On the other hand,
Proposition \ref{prop:p1p2_tensor_alg_onto}(b) is not a special case
of Proposition \ref{prop:tensor_algebra_surjective}(b).

The following proposition is similar to Proposition
\ref{prop:p1p2_ring_is_tensor_ring_of_species}.

\begin{proposition}\label{prop:p1_ring_is_tensor_ring_of_species}
Let $\Lambda$ be a finite dimensional basic split $k$-algebra, let
$\tilde{\mathcal{S}}_\Lambda$ be the enlarged species of $\Lambda$,
and let $\tilde{J}$ be the ideal $\oplus_{j \geq 1} \tilde{M}^j$ of
$T(\tilde{\mathcal{S}}_\Lambda)$. Then $\Lambda \simeq
T(\tilde{\mathcal{S}}_\Lambda)/\langle \rho \rangle$ for a set of
relations $\rho$ such that $\tilde{J}^{\rl (\Lambda)} \subset
\langle \rho \rangle \subset \tilde{J}$.
\end{proposition}
\begin{proof}
\sloppy The proof is similar to the proof of Proposition
\ref{prop:p1p2_ring_is_tensor_ring_of_species}.  First recall that
$T(\tilde{\mathcal{S}}_\Lambda) = T(\Lambda/ \mathfrak{r}, \tilde{M}) =
\tilde{T}(\Lambda / \mathfrak{r}, \mathfrak{r} / \mathfrak{r}^2)$.
From Proposition \ref{prop:tensor_algebra_surjective} we have an
epimorphism $\tilde{f} \colon T(\tilde{\mathcal{S}}_\Lambda) =
\tilde{T}(\Lambda / \mathfrak{r}, \mathfrak{r} / \mathfrak{r}^2) \to
\Lambda$ with $\oplus_{j \geq \rl (\Lambda)} \tilde{M}^j \subset
\ker \tilde{f} \subset \oplus_{j \geq 1} \tilde{M}^j$.  It is
possible to find a set of relations $\rho = \{ \sigma_t \}_{t\in T}$
in $\ker \tilde{f}$ which generates $\ker \tilde{f}$ as an ideal in
$T(\tilde{\mathcal{S}}_\Lambda)$. Then $\tilde{J}^{\rl (\Lambda)} =
\oplus_{j \geq \rl (\Lambda)} \tilde{M}^j \subset \langle \rho
\rangle \subset \oplus_{j \geq 1} \tilde{M}^j = \tilde{J}$. Since
$T(\tilde{\mathcal{S}}_\Lambda)/\langle \rho \rangle \simeq
\tilde{T}(\Lambda / \mathfrak{r}, \mathfrak{r} / \mathfrak{r}^2) /
\ker \tilde{f} \simeq \Lambda$, we are done.
\end{proof}

Using Corollary
\ref{cor:species_with_rel_is_tensor_with_ideal} and Proposition
\ref{prop:p1_ring_is_tensor_ring_of_species} we get the following corollary.

\begin{corollary}
Let $\Lambda$ be a finite dimensional basic split $k$-algebra.  Then
the category $\mod \Lambda$ is equivalent to the category $\rep
(\tilde{\mathcal{S}}_\Lambda, \rho)$ where $\rho$ is a set of
relations such that $\tilde{J}^{\rl (\Lambda)} \subset \langle \rho
\rangle \subset \tilde{J}$.
\end{corollary}

In the case when $\Lambda$ is $\mathfrak{r}$-split, then $\Lambda
\simeq T(\mathcal{S}_\Lambda)/\langle \rho \rangle$ where $\langle
\rho \rangle$ is an \emph{admissible ideal}, i.e.\ there exists a natural number $n \geq 2$ such that $J^n \subset \langle \rho \rangle \subset J^2$. On the other hand, if
$\Lambda$ is only split the ideal $\langle \rho \rangle$ is no
longer admissible in general. This difference will play an important role when investigating hereditary algebras in the next sections.

\section{Hereditary basic finite dimensional split algebras}

In this section we will continue looking at basic finite dimensional split algebras that are not assumed to be $\mathfrak{r}$-split, but with the extra assumption that they are hereditary. In contrast to the $\mathfrak{r}$-split algebras, the species of these hereditary algebras might have non-zero relations. However, the relations are corresponding to subquivers for which the underlying graph contains cycles. We start by introducing these subquivers.

A quiver consisting of an arrow $i \to j$ together with a finite
number (possibly zero) of paths between $i$ and $j$ of length
greater than one will be called a \emph{canonical
quiver}\index{canonical quiver}.
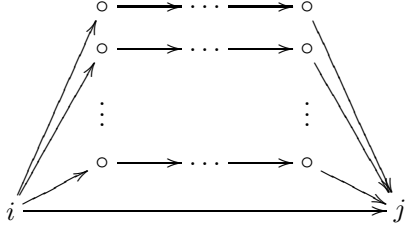
$$\xymatrix@R=0.2cm{ & \circ \ar[r] & \dots \ar[r] & \circ \ar[ddddr] & \\ & \circ \ar[r] & \dots \ar[r] & \circ \ar[dddr] & \\ & \vdots &  & \vdots & \\ & \circ \ar[r] & \dots \ar[r] & \circ \ar[dr] & \\ i \ar[rrrr] \ar[ur] \ar[uuur] \ar[uuuur]  &  &  &  & j }$$
We call the arrow $i \to j$ in a canonical quiver the
\emph{specifying arrow of the canonical quiver}\index{specifying
arrow}\index{canonical quiver!specifying arrow}. If $Q$ is a finite
quiver without oriented cycles or double arrows, and $i \to j$ is
the specifying arrow of a canonical quiver $Q' \subset Q$, then
there exists a largest canonical quiver $Q'' \subset Q$ having $i
\to j$ as its specifying arrow. We say that $Q''$ is \emph{the}
canonical quiver corresponding to $i \to j$ in $Q$. Note that a
canonical quiver can be equal to its specifying arrow.  If
$Q_{\tilde{\mathcal{S}}}$ is the quiver of an enlarged species
$\tilde{\mathcal{S}}_\Lambda$ of a finite dimensional hereditary
basic split algebra $\Lambda$, then $Q_{\tilde{\mathcal{S}}}$ is
finite without oriented cycles and without double arrows, so in this
setting we always have a largest canonical quiver corresponding to any given
arrow.

Let $\mathcal{S} = (D_i, {_j M_i})_{i,j \in I}$ be a species for
which the underlying quiver $Q_\mathcal{S}$ is finite and without
oriented cycles, and let ${_j \sigma_i} = g_1 + \cdots + g_n \in {_j
\mathcal{M}_i} = D_i T(\mathcal{S}) D_j$ be a relation. If ${_j
\sigma_i} \not\in J^2$ where $J = \oplus_{i \geq 1} M^{(i)}$, there
must be an arrow $i \to j \in Q_{\mathcal{S}}$ and at least one $g_l
\in {_j M_i}$.  If there is a set $\{ g_l \}$
in ${_j M_i}$, say $\{g_1, \dots
g_m \}$ where $g_i \in {_j M_i}$ and $m \leq n$, letting $\sigma' = g'_1 + g_{m+1} + \cdots
g_n$ where $g'_1 = g_1 + \cdots g_m$ we see that
$T(\mathcal{S})/\langle \sigma \rangle \simeq T(\mathcal{S})/\langle
\sigma' \rangle$ since $\sigma = \sigma'$. Hence we can always reduce to the case with only one $g_l \in {_j M_i}$

If $\mathcal{S} = (D_i, {_j M_i})_{i,j \in I}$ is a species, where
the underlying quiver $Q_\mathcal{S}$ is finite and without oriented
cycles, and ${_j \sigma_i} = g_1 + \cdots + g_n$ is a relation where
$g_l \in {_j M_i}$ while $g_k \not\in {_j M_i}$ for all $k \not= l$,
then we call the relation ${_j \sigma_i}$ a \emph{canonical
relation}\index{canonical relation}.  By renumbering we will always
assume $l = 1$ for a canonical relation. If $n > 1$ we call $\sigma$
a \emph{strong canonical relation}\index{canonical
relation!strong}\index{strong canonical!relation}. We call $\rho$ a
\emph{canonical set of relations}\index{canonical set of relations}
and $\langle \rho \rangle$ a \emph{canonical ideal}\index{canonical
ideal} of $T(\mathcal{S})$ if $\rho = \{ \sigma_t \}_{t \in T}$ such
that
\begin{itemize}
\item[(i)] all $\sigma_t$ are canonical relations
\item[(ii)] if $\{ \sigma_t \}_{t \in T'}$ is the set of relations in $\rho$ which
start in $i$ and end in $j$, and $\{ g^t_1 \}_{t \in T'}$ is the
corresponding set of summands which are elements of $_j M_i$, then
$\langle g^t_1 \rangle \cap \langle g^{t'}_{1} \rangle = (0)$ for
all $t \not= t'$
\end{itemize}
If all the canonical relations are strong, we call $\rho$ a
\emph{strong canonical set of relations}\index{canonical set of
relations!strong}\index{strong canonical!set of relations} and
$\langle \rho \rangle$ a \emph{strong canonical
ideal}\index{canonical ideal!strong}\index{strong canonical!ideal}.

\begin{lemma}\label{lemma:species_with_canonical_rel_is_hereditary}
Let $\mathcal{S}$ be a species for which the underlying quiver
$Q_\mathcal{S}$ is finite and without oriented cycles, and let
$\rho$ be a canonical set of relations. Then $T(\mathcal{S}) /
\langle \rho \rangle$ is hereditary.
\end{lemma}
\begin{proof}
\sloppy Let $\Lambda = T(\mathcal{S}) / \langle \rho \rangle$ be as
described in the lemma, then every indecomposable projective
$\Lambda$-module is of the form $P_i \simeq \Lambda D_i$ where
$\Lambda /\mathfrak{r} \simeq \oplus_{i \in I} D_i$. We want to show
that the radical $\mathfrak{r} D_i$ of every projective module
$\Lambda D_i$ is again projective.  Assume $\{{_j \sigma^t_i} \}_{t
\in {_j T_i}}$ is the complete set of (canonical) relations starting
in $i$ and ending in $j$, and let ${_j N_i} = {_j M_i} / \langle \{
g^t_1 \}_{t \in {_j T_i}} \rangle$ where ${_j \sigma^t_i} = g^t_1 +
g^t_2 + \dots g^t_{n_l}$ and $g^t_1 \in {_j M_i}$. Let $I'$ be the
set of vertices $j \in I$ for which ${_j M_i} \not= (0)$, and let
$I' = I_c \cup I_n$ be the disjoint union where $j \in I_c$ if there
exists a (canonical) relation ${_j \sigma_i}$, and $j \in I_n$ if no
such relation exists. Then
\begin{align*}\mathfrak{r} D_i & = (\rad T(\mathcal{S}) D_i )/ (\langle \rho
\rangle D_i) \\ & \simeq (\oplus_{j \in I_c} (T(\mathcal{S}) {_j
N_i}) / (\langle \rho \rangle {_j N_i})) \oplus (\oplus_{j \in I_n}
(T(\mathcal{S}) {_j M_i}) / (\langle \rho \rangle {_j M_i})) \\ &
\simeq (\oplus_{j \in I_c} (\dim_{D_j} {_j N_i}) (T(\mathcal{S})
D_j) / (\langle \rho \rangle D_j)) \oplus (\oplus_{j \in I_n}
(\dim_{D_j} {_j M_i}) (T(\mathcal{S}) D_j) / (\langle \rho \rangle
D_j)) \\ & = (\oplus_{j \in I_c} (\dim_{D_j} {_j N_i}) \Lambda D_j)
\oplus (\oplus_{j \in I_n} (\dim_{D_j} {_j M_i}) \Lambda D_j) \\ &
\simeq (\oplus_{j \in I_c} (\dim_{D_j} {_j N_i}) P_j) \oplus
(\oplus_{j \in I_n} (\dim_{D_j} {_j M_i}) P_j)
\end{align*}

This shows that $\mathfrak{r} D_i$ is projective, hence
$\mathfrak{r} = \oplus_{i \in I} \mathfrak{r} D_i$ is projective,
which implies that $\Lambda$ is hereditary \cite[Theorem 2.35]{La2}.
\end{proof}

Let $\Lambda$ be a finite dimensional hereditary basic split
$k$-algebra, where $\Lambda \simeq
T(\tilde{\mathcal{S}}_\Lambda)/\langle \rho \rangle$ and
$Q_{\tilde{\mathcal{S}}_\Lambda}$ is the underlying quiver of
$\tilde{\mathcal{S}}_\Lambda$.  From the discussion after
Proposition \ref{prop:p1p2_ring_is_tensor_ring_of_species} we see
that we can choose the set of relations $\rho = \{ \sigma_t \}_{t
\in T}$ finite, therefore we can find a finite minimal set of
relations which generate $\langle \rho \rangle$. The next
proposition reveals that even in the non $\mathfrak{r}$-split case
we find interesting information on the ideal $\langle \rho \rangle$.

\begin{proposition}\label{prop:relation_is_arrow}
Let $\Lambda$ be a finite dimensional hereditary basic split
$k$-algebra, where $\Lambda \simeq
T(\tilde{\mathcal{S}}_\Lambda)/\langle \rho \rangle$ and $\rho$ is a
finite set of relations. If $\rho = \{ \sigma_1, \dots , \sigma_m
\}$ is a minimal set, then $\rho$ is a canonical set of relations.
\end{proposition}
\begin{proof}
Let $\rho = \rho' \cup \rho''$ where $\rho'$ is the subset of $\rho$
consisting of canonical relations, and $\rho'' = \rho \setminus
\rho'$.  If there exists a pair $\sigma_t, \sigma_{t'} \in \rho'$ of
relations starting in $i$ and ending in $j$ such that $\langle g^t_1
\rangle \cap \langle g^{t'}_1 \rangle \not= (0)$, then $g^{t'}_1 =
d_j g^t_1 d_i$ for $d_i \in D_i$ and $d_j \in D_j$, so by
substitution we can find $\sigma'_{t'} \in \tilde{J}^2$ such that
$T(\tilde{\mathcal{S}}_\Lambda)/\langle \rho \rangle \simeq
T(\tilde{\mathcal{S}}_\Lambda)/\langle \sigma'_{t'}, \rho \setminus
\sigma_{t'} \rangle$.  By repeating this process we find a set
$\varrho = \varrho' \cup \varrho''$ of relations where every element
of $\varrho'$ is canonical, every element of $\varrho''$ sits inside
$\tilde{J}^2$, and for every pair $\sigma_t, \sigma_{t'} \in
\varrho'$ we have $\langle g^t_1 \rangle \cap \langle g^{t'}_1
\rangle = (0)$, i.e. $\varrho'$ is a canonical set of relations. Let
$\Lambda' \simeq T(\tilde{\mathcal{S}}_\Lambda)/\langle \varrho'
\rangle$, then $\Lambda'$ is hereditary from Lemma
\ref{lemma:species_with_canonical_rel_is_hereditary}.

Now \begin{align*}\Lambda & \simeq
T(\tilde{\mathcal{S}}_\Lambda)/\langle \varrho \rangle \simeq
\Lambda' / (\langle \varrho \rangle/\langle \varrho' \rangle) \\ &
\simeq \Lambda'
 / (\langle \varrho'' \rangle/ (\langle \varrho''
\rangle \cap \langle \varrho' \rangle))\end{align*}  Since $\langle
\varrho'' \rangle \subset \tilde{J}^2$ we have $\langle \varrho''
\rangle / (\langle \varrho'' \rangle \cap \langle \varrho' \rangle)
\subset \tilde{J}^2 / \langle \varrho' \rangle \simeq \rad^2
\Lambda'$ as left $\Lambda'$-modules. Lemma
\ref{lemma:ideal_in_square_radical} implies $\langle \varrho''
\rangle = 0$ since $\Lambda$ was assumed to be hereditary, hence
$\varrho = \varrho'$. Since $\rho'' \subset \varrho'' = \emptyset$,
we have $\rho = \rho'$. Moreover the existence of a pair $\sigma_t,
\sigma_{t'} \in \rho'$ of relations starting in $i$ and ending in
$j$ such that $\langle g^t_1 \rangle \cap \langle g^{t'}_1 \rangle
\not= (0)$ implies $\varrho'' \not= \emptyset$ since $\rho$ was
assumed to be minimal, a contradiction. Hence $\rho$ is a canonical
set of relations.
\end{proof}

\begin{proposition}\label{prop:only_specifying_relations}
Let $\Lambda$ be a finite dimensional hereditary basic split
$k$-algebra. Then $\Lambda \simeq T(\mathcal{S}_m)/\langle \rho
\rangle$ where $\mathcal{S}_m$ is a species unique up to isomorphism for which the
underlying quiver $Q_{\mathcal{S}_m}$ is finite and without oriented
cycles, and $\langle \rho \rangle$ is a strong canonical ideal.
\end{proposition}
\begin{proof}
From Proposition \ref{prop:p1_ring_is_tensor_ring_of_species} we
know that $\Lambda \simeq T(\tilde{\mathcal{S}}_\Lambda)/\langle
\rho' \rangle$ where we can choose the set $\rho'$ to be finite, and
from Lemma \ref{prop:relation_is_arrow} we know that $\rho'$ is a
canonical set of relations.

Let $\sigma = g_1 + \cdots + g_n \in \rho$ be a relation where $g_1
\in {_j \tilde{M}_i}$.  If the largest canonical quiver containing
$i \to j$ is the trivial canonical quiver $i \to j$ itself, then
${_j \tilde{\mathcal{M}}_i} = {_j \tilde{M}_i}$ and ${_j f_i} \colon
{_j \tilde{\mathcal{M}}_i} \otimes_{D_i} V_i \to V_j$ is just the
morphism ${_j \phi_i} \colon {_j \tilde{M}_i} \otimes_{D_i} V_i \to
V_j$. This implies that $\sigma = g_1 \in {_j \tilde{M}_i}$, so
$\langle \sigma \rangle$ generated as a $D_j - D_i$-bimodule is a
submodule of ${_j \tilde{M}_i}$. If we let ${_j \tilde{N}_i} = {_j
\tilde{M}_i} / \langle \sigma \rangle$, we can define a species
$\mathcal{S}_\sigma$ by substituting the bimodule ${_j \tilde{M}_i}$
in $\tilde{\mathcal{S}}_\Lambda$ with ${_j \tilde{N}_i}$. From the
construction of the tensor algebra we see that
$T(\tilde{\mathcal{S}}_\Lambda)/\langle \rho' \rangle \simeq
T(\mathcal{S}_\sigma)/\langle \rho' \setminus \sigma \rangle$.

If we use this process repeatedly, we can remove all relations in
$\rho'$ that are not strong canonical relations, and end up with a
unique (up to isomorphism) species $\mathcal{S}_m$ together with a set $\rho$ consisting
of the strong canonical relations in $\rho'$ such that $\Lambda
\simeq T(\mathcal{S}_m)/\langle \rho \rangle$.
\end{proof}

\begin{theorem}\label{thm:fin_dim_alg_eq}
Let $\Lambda$ be a finite dimensional basic split $k$-algebra. Then
the following are equivalent
\begin{itemize}
\item[(i)] $\Lambda$ is hereditary.
\item[(ii)] $\Lambda \simeq T(\mathcal{S}) / \langle \rho \rangle$ where
$\mathcal{S}$ is a species for which the underlying quiver
$Q_\mathcal{S}$ is finite and without oriented cycles and $\langle
\rho \rangle$ is a canonical ideal.
\item[(iii)] $\Lambda \simeq T(\mathcal{S}_m) / \langle \rho' \rangle$ where
$\mathcal{S}_m$ is a species unique up to isomorphism for which the underlying quiver
$Q_{\mathcal{S}_m}$ is a subquiver of $Q_\mathcal{S}$ and $\langle \rho'
\rangle$ is a strong canonical ideal.
\end{itemize}
\end{theorem}
\begin{proof}
This follows from Lemma
\ref{lemma:species_with_canonical_rel_is_hereditary}, Proposition
\ref{prop:relation_is_arrow}, and Proposition
\ref{prop:only_specifying_relations}. Note that $Q_{\mathcal{S}_m}$ might be a strict subquiver if a bimodule ${_j \tilde{M}_i}$ is replaced by a bimodule ${_j \tilde{N}_i}=(0)$.
\end{proof}

The next theorem sums up the results in this section similarly to
what Theorem \ref{theorem:all_for_P1_P2} did for Section
\ref{section:perfect_field}.

\begin{theorem}
\sloppy Let $\Lambda$ be a finite dimensional basic split
$k$-algebra, let $\tilde{\mathcal{S}}_\Lambda = (D_i, {_j
\tilde{M}_i})_{i,j \in I}$ be the enlarged species of $\Lambda$, let
$Q_{\tilde{\mathcal{S}}_\Lambda}$ be the underlying quiver of
$\tilde{\mathcal{S}}_\Lambda$, and let $\tilde{J}$ be the ideal
$\oplus_{i \geq 1} \tilde{M}^{(i)}$ in $\tilde{T} (\Lambda /
\mathfrak{r}, \mathfrak{r} / \mathfrak{r}^2)$ where $\tilde{M} =
\oplus_{i,j \in I} {_j \tilde{M}_i}$. Then the following hold:
\begin{itemize}
\sloppy \item[(a)] Let $\{ r_1, r_2, \dots, r_m \}$ be elements in
$\mathfrak{r}$ such that their images $\{ \bar{r}_1, \bar{r}_2,
\dots, \bar{r}_m \}$ in $\mathfrak{r} / \mathfrak{r}^2$ generate
$\mathfrak{r} / \mathfrak{r}^2$ as a $\Lambda / \mathfrak{r} -
\Lambda / \mathfrak{r}$-bimodule.  Then $\{ D_1, D_2,  \dots, D_n,
r_1, r_2, \dots, r_m \}$ generate $\Lambda$ as a $k$-algebra, where
$\Lambda / \mathfrak{r} \simeq \oplus_{i=1}^n D_i$.
\item[(b)] There is an onto $k$-algebra homomorphism $\tilde{f} \colon
\tilde{T} (\Lambda / \mathfrak{r}, \mathfrak{r} / \mathfrak{r}^2)
\to \Lambda$ such that $\oplus_{j \geq \rl (\Lambda)}
\tilde{M}^{(j)} \subset \ker \tilde{f} \subset \oplus_{j \geq 1}
\tilde{M}^{(j)}$.
\item[(c)] $\Lambda \simeq T(\tilde{\mathcal{S}}_\Lambda)/\langle
\rho \rangle$ where $\rho$ is a set of relations on
$\tilde{\mathcal{S}}$ with $\tilde{J}^{\rl (\Lambda)} \subset
\langle \rho \rangle \subset \tilde{J}$.
\item[(d)] The category $\mod \Lambda$ is equivalent to
the category $\rep (\tilde{\mathcal{S}}_\Lambda, \rho)$, where
$\rho$ is a set of relations on $\tilde{\mathcal{S}}_\Lambda$ with
$\tilde{J}^{\rl (\Lambda)} \subset \langle \rho \rangle \subset
\tilde{J}$.
\item[(e)] $\Lambda$ is hereditary if and only if $Q_{\tilde{\mathcal{S}}_\Lambda}$ is finite without oriented
cycles and $\Lambda$ is isomorphic to
$T(\tilde{\mathcal{S}}_\Lambda) / \langle \rho \rangle$ for a
canonical ideal $\rho$.
\item[(f)]
$\Lambda$ is hereditary if and only if $\Lambda$ is isomorphic to
$T(\mathcal{S}_m) / \langle \rho \rangle$ for a species
$\mathcal{S}_m$ and a strong canonical ideal $\rho$ where the
underlying quiver $Q_{\mathcal{S}_m}$ of $\mathcal{S}_m$ is finite
and without oriented cycles.
\end{itemize}
\end{theorem}

Let $\Lambda$ be a finite dimensional hereditary basic split
$k$-algebra, and let $\Lambda \simeq T(\mathcal{S}_m)/\langle \rho
\rangle$ where $\mathcal{S}_m = (D_i, {_j N_i})_{i,j \in I}$ and
$\rho$ are as described in Proposition
\ref{prop:only_specifying_relations}, hence $\rho$ consists of
strong canonical relations.  We want to investigate the species
$\mathcal{S}_m$. As before, let $\Lambda / \mathfrak{r} \simeq
\oplus D_i = D$ and $\mathfrak{r} / \mathfrak{r}^2 \simeq
\oplus_{i,j \in I} ({_j M_i}) = M$.  Let $i \to j$ be a trivial
canonical quiver.  We obviously have an isomorphism $D_j \Lambda D_i
\simeq D_j (T(\mathcal{S}_m)/\langle \rho \rangle) D_i$. What is
interesting is that $D_j \langle \rho \rangle D_i = (0)$ since $i
\to j$ was a trivial canonical quiver and $\rho$ only consists of
strong canonical relations. Therefore $D_j \Lambda D_i \simeq D_j
(T(\mathcal{S}_m)/\langle \rho \rangle) D_i \simeq D_j
T(\mathcal{S}_m) D_i$. 

If the underlying graph of $\mathcal{S}_m$ is a tree, i.e.\ it does not contain any cycles, then all canonical quivers in $\mathcal{S}_m$ must be trivial. This yields the following corollary.

\begin{corollary}
If the underlying graph of
$\mathcal{S}_m$ is a tree, then $\Lambda \simeq
T(\mathcal{S}_m) \simeq T(\mathcal{S}_\Lambda)$.
\end{corollary}

This corollary can also be deduced from the following observation \cite{DR3}:
Let $\Lambda$ be a finite dimensional hereditary basic split
$k$-algebra, and let $Q_{\mathcal{S}_\Lambda}$ be the underlying
quiver of the species associated to $\Lambda$.  Assume
$Q_{\mathcal{S}_\Lambda}$ is a tree. Then for any pair $i,j \in I$,
either $D_j \mathfrak{r}^2 D_i$ or $D_j \mathfrak{r} /
\mathfrak{r}^2 D_i$ must be zero.  Then obviously
$$0 \to D_j \mathfrak{r}^2 D_i \to D_j \mathfrak{r} D_i \to D_j
\mathfrak{r} / \mathfrak{r}^2 D_i \to 0$$ splits as a sequence of
$D_j-D_i$-modules via any $k$-algebra homomorphism $\epsilon \colon
\Lambda / \mathfrak{r} \to \Lambda$ such that $\pi \epsilon \simeq
\id_{\Lambda / \mathfrak{r}}$ for the natural projection $\pi \colon
\Lambda \to \Lambda / \mathfrak{r}$.  This implies that $$0 \to
\mathfrak{r}^2 \to \mathfrak{r} \to \mathfrak{r} / \mathfrak{r}^2
\to 0$$ splits as a sequence of $\Lambda / \mathfrak{r} - \Lambda /
\mathfrak{r}$-bimodule via $\epsilon$. Hence $\Lambda$ is
$\mathfrak{r}$-split, so $\Lambda \simeq T(\mathcal{S}_\Lambda)$.

\begin{example}

We will end this article with an example of a hereditary species
containing canonical relations. This example is motivated by
\cite{DR3}.  Let $k = \mathbb{F}_2 (t^2) = \{ \frac{f}{g} \mid f,g
\in \mathbb{F}_2[t^2], g\not= 0 \}$ be the field of rational
functions with indeterminate $t^2$ over the ground field
$\mathbb{F}_2$, where $\mathbb{F}_2$ is the Galois field consisting
of two elements. The field $k$ is not perfect. Let $K = \mathbb{F}_2
(t)$. We define a morphism $\delta \colon K \to K$ by using the
usual derivation with respect to $t$.  Now $\delta (f) = 0$ for $f
\in k \subset K$ due to the fact that $\Char \mathbb{F}_2 = 2$. Let
$M$ be the set $M = \{ (f,g) \mid f,g \in K \}$ where we define a
$K-K$-bimodule structure on $M$ by letting $a(f,g) = (af, ag)$ and
$(f,g)b = (fb, gb + f \delta(b))$ for $(f,g) \in M$ and $a,b \in K$.
The species $\mathcal{S}$ is given by the following diagram
$$\xymatrix{ & K \ar[dr]^K & \\ K \ar[ur]^K \ar[rr]_M & & K}$$
where the underlying quiver $Q_\mathcal{S}$ of $\mathcal{S}$ is
$$\xymatrix{ & 1 \ar[dr] & \\ 2 \ar[ur] \ar[rr] & & 0}$$
Let $\sigma$ be the relation $((1,0),0) - (0,1 \otimes_K 1) \in {_0
\mathcal{M}_2} = M \oplus (K \otimes_K K)$, and let $\rho$ be the
set of relations consisting only of $\sigma$. Then $(\mathcal{S},
\rho)$ is a species with relations, and $\rho$ is a set of strong
canonical relations. The tensor ring $T(\mathcal{S}) / \langle \rho
\rangle$ is isomorphic to the matrix ring
$$\Lambda = \begin{pmatrix}
K & 0 & 0 \\
K & K & 0 \\
M & K & K
\end{pmatrix}$$
where multiplication is normal matrix multiplication except for the
following
$$\begin{pmatrix}
0 & 0 & 0 \\
0 & 0 & 0 \\
0 & a & 0
\end{pmatrix} \begin{pmatrix}
0 & 0 & 0 \\
b & 0 & 0 \\
0 & 0 & 0
\end{pmatrix} = \begin{pmatrix}
0 & 0 & 0 \\
0 & 0 & 0 \\
(ab,0) & 0 & 0
\end{pmatrix}$$
From \cite[Corollary 2]{DR3} we know that this is a hereditary
finite dimensional $k$-algebra. We have $\Lambda / \mathfrak{r}
\simeq
\begin{pmatrix}
K & 0 & 0 \\
0 & K & 0 \\
0 & 0 & K
\end{pmatrix}$, and since $$\epsilon \colon \begin{pmatrix}
K & 0 & 0 \\
0 & K & 0 \\
0 & 0 & K
\end{pmatrix} \hookrightarrow \begin{pmatrix}
K & 0 & 0 \\
K & K & 0 \\
M & K & K
\end{pmatrix}$$ is a $k$-algebra homomorphism such that $\pi \epsilon \simeq
\id_{\Lambda / \mathfrak{r}}$ for the natural projection $\pi \colon
\Lambda \to \Lambda / \mathfrak{r}$, we see that $\Lambda$ is split.
Look at the sequence
$$0 \to D_0 \mathfrak{r}^2 D_2 \to D_0 \mathfrak{r} D_2 \to D_0
\mathfrak{r} / \mathfrak{r}^2 D_2$$ of $\Lambda / \mathfrak{r} -
\Lambda / \mathfrak{r}$-bimodules via $\epsilon$.  This is the
sequence $$0 \to K \to M \to K \to 0$$ where the first morphism is
given by $a \mapsto (a,0)$ for $a \in K$, and the last morphism is
given by $(a,b) \mapsto b$ for $(a,b) \in M$. Since $M \not\simeq
K^2$ this sequence does not split, hence $\Lambda$ is not a
$\mathfrak{r}$-split algebra. Hence $\Lambda$ is an
example of an algebra that satisfies the assumptions in Theorem
\ref{thm:fin_dim_alg_eq}, but not the
assumptions in Theorem
\ref{thm:fin_dim_alg_r-split_eq}. \exend
\end{example}

\end{document}